\documentclass[A4paper,12pt]{article}

\usepackage[top=4cm, bottom=4cm, left=2.5cm, right=2.5cm]{geometry}
\usepackage[dvipdfmx]{graphicx}
\usepackage[dvipdfmx]{color}
\usepackage{amsmath,amsthm}
\usepackage{amssymb}
\usepackage{amsfonts}
\usepackage{comment}
\usepackage{ascmac}
\usepackage{color}
\usepackage{mathtools}
\usepackage{latexsym,amssymb}
\usepackage{amsthm} 
\usepackage[mathscr]{eucal}
\usepackage{color}
\usepackage[abbrev]{amsrefs}
\usepackage[T1]{fontenc}

\usepackage{latexsym}
\usepackage{amssymb}
\usepackage{amsfonts}
\usepackage{enumerate}
\usepackage{mathrsfs}
\usepackage{comment}
\usepackage{mathtools}

\newtheorem{theo}{Theorem}
\newtheorem{prop}{Proposition}[section]

\newtheorem{lem}{Lemma}[section]
\newtheorem{defi}{Definition}[section]
\newtheorem{rem}{Remark}[section]

\mathtoolsset{showonlyrefs=true}

\makeatletter

\@addtoreset{equation}{section}
\makeatother

\newcommand{\LOG}[1]{ \log \left( e + \frac{1}{ #1 } \right) }
\newcommand{\zygm}[2]{ \mathcal{L}^{ #1 ,\infty}\left(\log\mathcal{L}\right)^{ #2 } }
\newcommand{\unifz}[2]{ \mathcal{L}_{\rm ul}^{ #1 ,\infty}\left(\log\mathcal{L}\right)^{ #2 } }

\title{On a Fujita critical time-fractional semilinear heat equation in the uniformly local weak Zygmund type space}
\date{}
\author{Mizuki Kojima
\thanks{The author was supported by JST SPRING, Japan Grant Number JPMJSP2106.}
\thanks{E-mail: \texttt{kojima.m.aq@m.titech.ac.jp}}	\\
		Department of Mathematics, Tokyo Institute of Technology.
		}

\begin{document}
	\maketitle
	\begin{abstract}

	In this paper, we derive sufficient conditions on initial data for the
	local-in-time solvability of a time-fractional semilinear heat equation
	with the Fujita exponent in a uniformly local weak Zygmund type space.
	It is known that the time-fractional problem with the Fujita exponent in the scale
	critical space $L^1(\mathbb{R}^N)$ exhibits the local-in-time solvability in contrast to the unsolvability of the Fujita critical classical semilinear heat equation.
	Our new sufficient conditions take into account the fine structure of
	singularities of the initial data, in order to show a natural
	correspondance between the time-fractional and the classical case for the
	local-in-time solvability.
	We also apply our arguments to life span estimates for some typical initial data. 
	
	
\end{abstract}

\section{Introduction}

In this paper, we consider the Cauchy problem for the time-fractional semilinear heat equation
\begin{equation}\label{eq. TFF}
\left\{
\begin{aligned}
{}^{C}\partial_t^{\alpha} u - \Delta u &= |u|^{p-1}u\ \ &&\mbox{in}\ (0,T)\times \mathbb{R}^N,\\
u(0)&=\varphi\ \ &&\mbox{in}\ \mathbb{R}^N
\end{aligned}
\right.
\end{equation}
where $N\ge 1$, $p>1$ and ${}^{C}\partial_t^{\alpha}$ denotes the Caputo derivative of order $\alpha\in (0,1)$:
\[
{}^{C}\partial_t^{\alpha}f(t):=\frac{1}{\Gamma(1-\alpha)}\frac{\partial}{\partial t}\int_{0}^{t}(t-s)^{-\alpha}\left( f(s)-f(0) \right)ds.
\]
In particular, we are concerned with the solvability of \eqref{eq. TFF} with the Fujita critical exponent $p=p_F:=1+2/N$. When $\alpha=1$, problem \eqref{eq. TFF} formally coincides with the classical Fujita-type problem in which the derivative with respect to time is the classical one:
\begin{equation}\label{eq. Fujita}
\partial_tu-\Delta u = |u|^{p-1}u, \quad t>0,\ x\in \mathbb{R}^N, \quad \quad u(0)=\varphi\ \mbox{in}\ \mathbb{R}^N,
\end{equation}
which has been extensively studied by many mathematicians since the pioneering work by Fujita \cite{Fujita66}. We just refer to the related studies \cite{Fujita66, Hayakawa73, KoSiTa77, BarasPierre85,BreCaz96,FHIL23,FujiIoku18,HisaIshige18,KozoYama94,Miya21,Wei80} and the comprehensive monograph \cite{QuitSoup19}. The Fujita exponent $p_F$ is the threshold of the global-in-time solvability of \eqref{eq. Fujita}. More precisely, the condition $1<p\le p_F$ implies the nonexistence of global-in-time nonnegative solutions of \eqref{eq. Fujita}, while $p_F<p$ guarantees the global-in-time solvability for appropriately small initial data.

The Fujita exponent also has a key role in the local-in-time solvability of problem \eqref{eq. Fujita}. Let $q_c:=N(p-1)/2$, which is a scale critical exponent of \eqref{eq. Fujita}. By the Fujita-Kato principle, it is important to consider the well-posedness in the scaling critical space. In fact, the problem \eqref{eq. Fujita} is locally-in-time well-posed in its scale critical space $L^{q_c}(\mathbb{R}^N)$, under the assumption $q_c>1$. However, the local-in-time solvability in $L^{q_c}(\mathbb{R}^N)$ fails when $q_c=1$.  Note that $q_c=1$ means that $p=p_F$. That is, there exists a nonnegative function $\varphi \in L^1(\mathbb{R}^N)$ such that, problem \eqref{eq. Fujita} with $p=p_F$ and $u(0)=\varphi\ge 0$ admits no local-in-time nonnegative solutions. See e.g.\ \cite{BarasPierre85,BreCaz96,HisaIshige18,Miya21}. Hence, the Fujita critical problem (i.e., problem \eqref{eq. Fujita} with $p=p_F$) exhibits unsolvability in its scale critical space $L^{1}(\mathbb{R}^N)$. Such a situation is often called \textit{doubly critical} (simultaneous occurrence of the Fujita critical and the scale critical situations).

We go back to the time-fractional problem \eqref{eq. TFF}. By Zhang and Sun \cite{ZhanSun15}, it is known that the Fujita exponent is common to \eqref{eq. TFF} and \eqref{eq. Fujita}, namely, the global-in-time solvability of \eqref{eq. TFF} fails if $1<p<p_F$, while $p_F\le p$ implies the existence of a global-in-time solution for appropriate initial data. We emphasize that a global solution can exist even if $p=p_F$. Moreover, problem \eqref{eq. TFF} is also scale critical in $L^{q_c}(\mathbb{R}^N)$, and its local-in-time solvability in its scale critical space is guaranteed for $q_c\ge 1$. See e.g.\ \cite{ZLS19,GMS22}. We must note that the local-in-time solvability of \eqref{eq. TFF} under the doubly critical situation is verified. For detailed topics for the time-fractional Fujita problem, see e.g.\ \cite{Podl99, GMS22, GalWar20, LuchYama19, Mainardi1994, ZhanSun15, ZLS19, OkaZhan23, HisaKojima24, CarFerNoe24} and references therein.

\begin{table}[h]\label{tab. comparison}
	\caption{Comparison of the two problems with $p=p_F$.}
	\centering
	\begin{tabular}{cc}
		\hline
		Problem & Local-in-time solvability in its scale critical space $L^1$ \rule[0mm]{0mm}{5mm}\\
		\hline \hline
		\eqref{eq. Fujita} with $p=p_F$ & 
		No nonnegative solutions for some $\varphi\in L^1(\mathbb{R}^N)$ with $\varphi\ge 0$.
		\rule[0mm]{0mm}{5mm}\\
		\eqref{eq. TFF} with $p=p_F$ & There exists a solution for all $ \varphi\in L^1(\mathbb{R}^N)$.
		\rule[0mm]{0mm}{5mm}\\  
		\hline
	\end{tabular}
\end{table}

Based on these significant differences as for the global and local-in-time solvability of Fujita critical problems, Hisa and the author \cite{HisaKojima24} investigated the \textit{collapse} of the solvability of \eqref{eq. TFF} with $p=p_F$ when $\alpha \nearrow 1$. In this research, a necessary and condition a sufficient one for the solvability of \eqref{eq. TFF} are obtained. Let $B(z;r)$ be a ball centered at $z\in\mathbb{R}^N$ with radius $r>0$.

\begin{itemize}
	\item (Necessity)\cite[Theorem~1.2]{HisaKojima24} Suppose that $\varphi\ge 0$ and problem \eqref{eq. TFF} with $p=p_F$ possesses a nonnegative solution in $(0,T)\times\mathbb{R}^N$ for some $T>0$.
	Then there exists a constant $\gamma_1=\gamma_1(N,\alpha)>0$ satisfying $\limsup_{\alpha \nearrow1} \gamma_1(N,\alpha)\in(0,\infty)$ such that
	\begin{equation}\label{eq. Hisa-Kojima necesarry}
	\sup_{z\in\mathbb{R}^N} \int_{B(z;\sigma)} \varphi(x)\,dx \le \gamma_1 \left(\int_{\sigma^{2/\alpha}/(16T)}^{1/4} t^{-\alpha} \, dt \right)^{-\frac{N}{2}}\ \ \text{for all}\ \ \sigma \in (0,T^{\frac{\alpha}{2}}).
	\end{equation}
	
	\item (Sufficiency)\cite[Theorem~1.4]{HisaKojima24} Let $p=p_F$ and $T>0$. Then there exists a constant $\gamma_2=\gamma_2(N,\alpha)>0$ satisfying $\limsup_{\alpha\nearrow1} \gamma_2(N, \alpha)\in(0,\infty)$ such that, 
	if $\varphi\ge 0$ satisfies
	\begin{equation}\label{eq. Hisa-Kojima sufficient}
	\sup_{z\in\mathbb{R}^N} \int_{B(z;T^{\frac{\alpha}{2}})} \varphi(x)\,dx\le  \gamma_2 (1-\alpha)^{\frac{N}{2}},
	\end{equation}
	then problem \eqref{eq. TFF} possesses a solution in $(0,T)\times\mathbb{R}^N$.
\end{itemize}

Through these arguments, in \cite{HisaKojima24}, the connection between the solvability of \eqref{eq. TFF} for any $L^1$ initial data, and the unsolvability of \eqref{eq. Fujita} for certain singular initial data, is characterized through life span estimates. We observe that the necessary condition \eqref{eq. Hisa-Kojima necesarry} formally converges to the corresponding condition for the classical Fujita type problem \eqref{eq. Fujita} as $\alpha \nearrow 1$. Indeed, for the classical problem \eqref{eq. Fujita} with initial data $\varphi\ge 0$ and $p=p_F$, suppose that there exists a solution in $(0,T)\times \mathbb{R}^N$ for $0<T<\infty$. Then, there exists a constant $\gamma=\gamma(N)$ such that $\varphi\ge 0$ satisfies
\begin{equation}\label{eq. Hisa-Ishige}
\sup_{z\in\mathbb{R}^N} \int_{B(z;\sigma)} \varphi(x)\,dx \le \gamma \left(
\log\left( e+ \frac{T^{ \frac{1}{2} }}{\sigma} \right) \right)^{-\frac{N}{2}}\ \ \text{for all}\ \ \sigma \in (0,T^{\frac{1}{2}}).
\end{equation}
See \cite{BarasPierre85, HisaIshige18} and see also \cite{IshiKawaOka2020, FHIL23, HIT18, LaiSie21, Takahashi16} for more related topics. It is easily observed that the limit of the right hand side of \eqref{eq. Hisa-Kojima necesarry} as $\alpha \nearrow 1$ coincides with the right hand side of \eqref{eq. Hisa-Ishige}.

In contrast to the reasonable correspondence between the necessary conditions \eqref{eq. Hisa-Kojima necesarry} and \eqref{eq. Hisa-Ishige}, the right hand side of the sufficient condition \eqref{eq. Hisa-Kojima sufficient} vanishes when $\alpha \nearrow 1$. That is, \eqref{eq. Hisa-Kojima sufficient} does not yield any reasonable result when $\alpha \nearrow 1$, although there exists $L^1(\mathbb{R}^N)$ initial data which satisfies local-in-time solvability of \eqref{eq. Fujita} with $p=p_F$. In fact, the optimal singularity of the initial data for the local-in-time solvability of \eqref{eq. Fujita} with $p=p_F$ has been identified. See e.g.\ \cite{HisaIshige18, Miya21}.
Therefore, the aim of this paper is to deduce a sufficient condition for the solvability of \eqref{eq. TFF} which yields reasonable meaning even when $\alpha\nearrow 1$. More precisely, we shall obtain a sufficient condition that corresponds to a known optimal sufficient condition of the local-in-time solvability of \eqref{eq. Fujita} when $\alpha \nearrow 1$. Moreover, as an application, we shall show that the solvability of \eqref{eq. TFF} corresponds to the solvability of \eqref{eq. Fujita} through sharp life span estimates for certain initial data, whereas \cite{HisaKojima24} shows the correspondence between the solvability of \eqref{eq. TFF} and the \textit{unsolvability} of \eqref{eq. Fujita}.

We shall explain a difficulty to investigate the solvability of \eqref{eq. TFF} with $p=p_F$. Although some previous works study the local-in-time solvability of \eqref{eq. Fujita} with singular initial data (e.g.\ \cite{HisaIshige18,Miya21}), these methods cannot be directly applied to \eqref{eq. TFF}. The main tool of those works is the supersolution method in which we need to verify that the certain function is a supersolution of the problem. However, due to the nonlocality of the time-fractional derivative, it is difficult to construct the supersolution of problem \eqref{eq. TFF} by the direct application of the previous method.

To overcome this difficulty, it seems that the contraction mapping argument is the suitable method since we can use the semigroup estimate, even for the time-fractional problem. Recently, Ioku, Ishige, and Kawakami \cite{IokuIshigeKawakami} constructed the framework, in which they introduced the functional space called the uniformly local weak Zygmund type space, to treat the Fujita critical problem \eqref{eq. Fujita} in the functional analytic method (see also \cite{FujiIshiKawa24,IshiKawaTakada24}). Their framework seems to be the suitable tool for the problem \eqref{eq. TFF} with $p=p_F$. However, the semigroup estimate that is obtained in \cite{IokuIshigeKawakami} is not sufficient in its application to \eqref{eq. TFF}, due to the nonlocality with respect to time. We shall give a more detailed explanation about this point before Theorem~\ref{theo. semigroup estimate}.
	\vskip\baselineskip
	
Before stating main results, we introduce some notation and definition. For $t>0$ and $x\in\mathbb{R}^N$, let $G_t(x)$ be the fundamental solution of the heat equation in $(0,\infty)\times\mathbb{R}^N$, that is
\[
G_t(x) := (4\pi t)^{-\frac{N}{2}}\exp\left(-\frac{|x|^2}{4t}\right).
\]
For $t>0$, $x\in\mathbb{R}^N$, and a function $\mu$ on $\mathbb{R}^N$, define
\[
\left( e^{t\Delta} \mu \right)(x) := \int_{\mathbb{R}^N} G_t( x-y) \mu (y)\,dy.
\]
Let
\begin{equation}\label{eq. defi of P and S}
P_{\alpha}(t)\mu := \int_0^\infty h_\alpha (\theta) e^{t^\alpha \theta \Delta}\mu d\theta, \quad S_{\alpha} (t)\mu := \int_0^\infty \theta h_\alpha (\theta ) e^{t^\alpha\theta \Delta} \mu d\theta,
\end{equation}
for $t>0$ and $0<\alpha<1$. Here, $h_\alpha$ is a probability density function so-called the Wright type function on $(0,\infty)$ which satisfies $h_\alpha(\theta)>0$ for $\theta>0$. In addition, we have
\begin{equation}\label{eq. integral of halpha}
\int_0^\infty \theta^{\delta} h_\alpha(\theta) \, d\theta = \frac{\Gamma(1+\delta)}{\Gamma(1+\alpha\delta)}<\infty \quad \mbox{for} \quad \delta>-1.
\end{equation}

We formulate the definition of solutions of problem \eqref{eq. TFF}. 


\begin{defi}
	\label{defi. solution}
	Let $0<T\le \infty$ and $u$ be a measurable and finite function almost everywehre in $(0,T)\times \mathbb{R}^N$. For given $\varphi$, we say that $u$ is a solution of \eqref{eq. TFF} in $(0,T)$ if $u$ satisfies
	\[
	u(t,x) = \left(P_\alpha(t)\varphi\right) (x) +\alpha\int_0^t (t-s)^{\alpha-1} \left(S_\alpha(t-s)|u(s)|^{p-1}u(s)\right)(x) \,ds
	\]
	for almost all $t\in(0,T)$ and $x\in\mathbb{R}^N$. If $u$ satisfies the above equality with $=$ replaced by $\ge$, then $u$ is said to be a supersolution of \eqref{eq. TFF} in $(0,T)$.
\end{defi}


Next, we define the weak Zygmund type space which is introduced in \cite{IokuIshigeKawakami} by
\[
\zygm{q}{\gamma}:= \left\{ f\in L^1_{\rm loc}; \| f\|_{\zygm{q}{\gamma}}<\infty \right\},
\]
where
\begin{equation}\label{eq. definition of Zygmund norm}
\|f\|_{\zygm{q}{\gamma}}:=\left\{
\begin{aligned}
& \sup_{s>0} \left[ \left( \LOG{s} \right)^{\gamma}\sup_{|E|=s} \int_{E} |f(x)|^{q}\,dx \right]^{\frac{1}{q}}\ \ &&\mbox{if}\ \ 1\le q<\infty,\\
&\|f\|_{L^{\infty}(\mathbb{R}^N)}\ \ &&\mbox{if}\ \ q=\infty.
\end{aligned}
\right.
\end{equation}
Moreover, let
\[
\|f\|_{q,\gamma;\rho}:= \sup_{z\in\mathbb{R}^N} \|f\chi_{B(z;\rho)}\|_{\zygm{q}{\gamma}}\quad \text{and}\quad \unifz{q}{\gamma}:= \left\{ f\in L^1_{\rm loc}; \|f\|_{q,\gamma;1}<\infty \right\},
\]
where $\chi_{E}$ is a characteristic function of a measureble set $E$. It is verified that $\zygm{q}{\gamma}$ and $\unifz{q}{\gamma}$ are Banach spaces (see \cite[Lemma~2.1]{IokuIshigeKawakami}).

Now we are ready to state our main results.

\begin{theo}\label{theorem 1}
	Let $p=p_F$. For given $\varphi \in \unifz{1}{\gamma}$ and $\gamma\ge 0$, one has the following: There exists $C=C(N,\gamma)$ such that
	\begin{equation}\label{eq. main vanish}
	\| \varphi\|_{1,\gamma;T^{\frac{\alpha}{2}}}\le C \alpha^{\frac{N}{2}} (1-\alpha)^{\frac{N}{2}} \left( \LOG{T} \right)^{\gamma}
	\end{equation}
	implies the existence of a solution $u$ of \eqref{eq. TFF} on $(0,T)$ which satisfies for any $\tilde{\gamma}\in [0, \gamma)$,
	\begin{equation}\label{eq. regularity(i)}
	\sup_{0<t<T} \|u(t)\|_{1,\gamma;T^{\frac{\alpha}{2}}}+ \sup_{0<t<T} t^{\frac{\alpha N}{2}\left(1-\frac{1}{p}\right) }\left( \LOG{t} \right)^{\gamma -\frac{\tilde{\gamma}}{p}} \| u(t)\|_{p,\tilde{\gamma};T^{\frac{\alpha}{2}}}\le C\| \varphi\|_{1,\gamma;T^{\frac{\alpha}{2}}},
	\end{equation}
	where $C=C(N, \gamma, \tilde{\gamma})$. Moreover, there exists $C=C(N,\gamma)$ such that
	\begin{equation}\label{eq. initial valure (i)}
	\lim_{t\searrow 0} \| u(t)- P_{\alpha}(t) \varphi\|_{1,\gamma;T^{\frac{\alpha}{2}}}\le C\alpha^{\frac{N}{2}}(1-\alpha)^{\frac{N}{2}} \left( \LOG{T} \right)^{\gamma} \left( \LOG{t} \right)^{\gamma(1-p)}.
	\end{equation}
\end{theo}

If we further restrict the singularity of the initial data so that the solvability of \eqref{eq. Fujita} is guaranteed, the $\alpha$-dependence of \eqref{eq. main vanish} is improved in the sense that the condition does not vanish as $\alpha \nearrow 1$. More precisely, we can deduce a sufficient condition that yields a reasonable meaning when $\alpha\nearrow 1$, as desired.

\begin{theo}\label{theorem 2}
	In addition to the same assumption as in Theorem~\ref{theorem 1}, assume further that $\gamma=N/2$. Then there exists $C=C(N)$ such that
	\begin{equation}\label{eq. main notvanish}
	\| \varphi\|_{1,\frac{N}{2}; T^{\frac{\alpha}{2}}} \le C \alpha^{\frac{N}{2}} \left(\log(e+2T) \right)^{-\frac{N}{2}}
	\end{equation}
	implies the existence of a solution of \eqref{eq. TFF} on $(0,T)$ which satisfies \eqref{eq. regularity(i)}. Moreover, there exists $C=C(N)$ such that
	\begin{equation}\label{eq. convergence gamma N/2}
	\| u(t)-P_{\alpha}(t)\varphi \|_{1,\gamma;T^{\alpha/2}} \le C \alpha^{\frac{N}{2}} (1-\alpha)^{-1} \left( \log\left( e+2T \right) \right)^{-\frac{N}{2}}\left( \LOG{t} \right)^{\gamma(1-p)}.
	\end{equation}
\end{theo}

Obviously, Theorem~\ref{theorem 1} is a generalization of \eqref{eq. Hisa-Kojima sufficient}, since \eqref{eq. main vanish} with $\gamma=0$ coincides with \eqref{eq. Hisa-Kojima sufficient}. On the other hand, the quantity in the right hand side of \eqref{eq. main notvanish} does not vanish as $\alpha\nearrow 1$. Note that
\[
\unifz{q}{\gamma_1}\subset \unifz{q}{\gamma_2}
\]
if $\gamma_1 >\gamma_2$. Therefore, the singularity $\gamma=N/2$ is regarded as the threshold of "non-vanishment" and "vanishment" as $\alpha \nearrow 1$. 

This significant difference of the $\alpha$-dependence inherits the treatment of the nonlinear term in the proofs of Theorems~\ref{theorem 1} and \ref{theorem 2}. For general $\varphi \in \unifz{1}{\gamma}$ and $\gamma \ge 0$, we essentially utilize the fact that $\alpha<1$, so arguments fail when $\alpha =1$. Therefore, the statement in Theorem~\ref{theorem 1} does not yield reasonable meaning when $\alpha \nearrow 1$. On the other hand, if we restrict the singularity of $\varphi$, we do not need to rely on the condition $\alpha<1$. Thus, the statement in Theorem~\ref{theorem 2} does not vanish even when $\alpha\nearrow 1$. More precisely, the key ingredient is Lemma~\ref{lem. log integral}. Although it comes from just simple calculations of integral, we are able to treat the nonlinear term in essentially different ways in the proofs of Theorems~\ref{theorem 1} and \ref{theorem 2}.

We compare Theorems~\ref{theorem 1} and \ref{theorem 2} with \cite[theorem~1.1]{IokuIshigeKawakami}, which states that the local-in-time solvability of \eqref{eq. Fujita} in the uniformly local weak Zygmund space is characterized as follows: Let $p=p_F$ and $T_*\in (0,\infty)$ be fixed. Then, there exists $C=C(T_*)>0$ such that, if $\varphi \in \unifz{1}{N/2}$ satisfies
	\begin{equation}\label{eq. IokuIshigeKawakami}
	\|\varphi \|_{1,\frac{N}{2};T^{\frac{1}{2}}}\le C
	\end{equation}
	for some $T\in (0,T_*]$, then \eqref{eq. Fujita} possesses a solution $u$ on $(0,T)$ which satisfies for $\tilde{\gamma}\in [0, N/2)$,
	\begin{equation}\label{eq. regularity (F)}
	\sup_{0<t<T} \|u(t)\|_{1,\frac{N}{2};T^{\frac{1}{2}}}+ \sup_{0<t<T} t^{\frac{ N}{2}\left(1-\frac{1}{p}\right) }\left( \LOG{t} \right)^{\frac{N}{2} -\frac{\tilde{\gamma}}{p}} \| u(t)\|_{p,\tilde{\gamma};T^{\frac{1}{2}}}\le C\| \varphi\|_{1,\frac{N}{2};T^{\frac{1}{2}}}
	\end{equation}
	where $C=C(N, \tilde{\gamma}, T^*)$. In particular, for any $\varphi\in \unifz{1}{N/2}$, there exists a local-in-time solution of \eqref{eq. Fujita}. Moreover, it follows that
	\begin{equation}\label{eq. initial valure (F)}
	\lim_{t\searrow 0} \| u(t)- e^{t\Delta} \varphi\|_{1,\tilde{\gamma};T^{\frac{1}{2}}} = 0.
	\end{equation}

Although \eqref{eq. Hisa-Kojima sufficient} vanishes as $\alpha\nearrow1$, the right hand side of condition \eqref{eq. main notvanish} does not vanish as $\alpha\nearrow 1$ as desired. Indeed, all initial data in $\unifz{1}{N/2}$ must satisfy the local-in-time solvability of the classical Fujita problem \eqref{eq. Fujita} thanks to \eqref{eq. IokuIshigeKawakami}. Therefore, Theorem~\ref{theorem 2} is regarded as the corresponding result to \cite[Theorem~1.1]{IokuIshigeKawakami}. On the other hand, the right hand side of condition \eqref{eq. main vanish} vanishes as $\alpha \nearrow 1$, and this fact reflects the collapse of the local-in-time solvability of \eqref{eq. TFF} for certain initial data which belongs to $\unifz{1}{\gamma}$ for $\gamma<N/2$. This point shall be discussed in detail in Section~\ref{section. Life span}, through life span estimates for typical initial data.

\begin{table}[h]
	\caption{Solvability for $\varphi \in \unifz{1}{\gamma}$.}
	\centering
	\begin{tabular}{ccc}
		\hline
		Problem & $\gamma\ge N/2$ & $\gamma<N/2$ \rule[0mm]{0mm}{5mm}\\
		\hline \hline
		\eqref{eq. Fujita} with $p=p_F$ & 
		Solvable & Unsolvable for some $\varphi\ge 0$.
		\rule[0mm]{0mm}{5mm}\\
		\eqref{eq. TFF} with $p=p_F$ & Solvable & Solvable, but \eqref{eq. main vanish} vanishes as $\alpha \nearrow 1$.
		\rule[0mm]{0mm}{5mm}\\  
		\hline
	\end{tabular}
\end{table}

We also mention the continuity at the initial time $t=0$. For \eqref{eq. Fujita}, the continuity at $t=0$ \eqref{eq. initial valure (F)} was proved in the norm of  $\unifz{1}{\tilde{\gamma}}$ for $\tilde{\gamma}<N/2$, which is weaker than that of the class $\unifz{1}{N/2}$ of initial data. However, for the time-fractional problem \eqref{eq. TFF}, the continuity at $t=0$ \eqref{eq. initial valure (i)} is guaranteed in the norm of $\unifz{1}{\gamma}$ which is the space $\varphi$ belongs to. This is a specific property for the time-fractional problem \eqref{eq. TFF}. As a matter of fact, this difference is suggested by \eqref{eq. convergence gamma N/2}, since the right hand side of \eqref{eq. convergence gamma N/2} goes to infinity as $\alpha\nearrow 1$.

	In Section~\ref{section. Life span}, we apply Theorems~\ref{theorem 1} and \ref{theorem 2} to life span estimates for the initial data 
	\begin{equation}
	f_{\beta}(x):= |x|^{-N} \left( \LOG{|x|} \right)^{-\frac{N}{2}-1-\beta}\quad \text{for}\quad -\frac{N}{2}<\beta<\infty.
	\end{equation}
	Note that $f_{\beta}\in \unifz{1}{\frac{N}{2}+\beta}$ (see Proposition~\ref{prop. fbeta}). Thus, if $\beta>0$, then \eqref{eq. Fujita} with $\varphi=f_{\beta}$ is locally-in-time solvable, whereas if $-N/2<\beta<0$, it is known that the unsolvability of \eqref{eq. Fujita} is deduced by the necessary condition \eqref{eq. Hisa-Ishige} (see \cite{BarasPierre85, HisaIshige18, Miya21}). Therefore, observations in Section~\ref{section. Life span} shall be divided into the following three cases:
	\begin{itemize}
		\item $\beta>0$ (Solvable even when $\alpha \nearrow 1$),
		
		\item $\beta=0$ (Threshold),
		
		\item $-\frac{N}{2}<\beta<0$ (Unsolvable when $\alpha\nearrow 1$).
	\end{itemize}

	Through life span estimates, we shall observe that the solvability of \eqref{eq. TFF} connects to the solvability of \eqref{eq. Fujita}, when $\varphi=f_{\beta}$ and $\beta>0$. In contrast, the solvability of \eqref{eq. TFF} connects to the unsolvability of \eqref{eq. Fujita} when $\beta<0$. See Theorems~\ref{theo. beta positive: kappa large}, \ref{theo. beta positive: kappa small} and Theorem~\ref{theo. beta negative} respectively. Furthermore, in the threshold case $\beta=0$, we also obtain a natural connection between \eqref{eq. TFF} and \eqref{eq. Fujita}. See Theorem~\ref{theo. beta zero}.

	We further apply Theorems~\ref{theorem 1} and \ref{theorem 2} to the life span estimate for certain initial data which is investigated by Lee, Ni \cite{LeeNi92}. We observe that the specific estimate holds to the time-fractional problem \eqref{eq. TFF}, and it connects to the classical one for \eqref{eq. Fujita}.

	The rest of the paper is organized as follows. In Section~\ref{section. Semigroup estimate}, we deduce a semigroup estimate in the uniformly local weak Zygmund space. The corresponding estimate is already studied in \cite[Proposition~3.2]{IokuIshigeKawakami}, but this argument cannot be directly applied to the problem \eqref{eq. TFF}, because of the restriction on the range of time. We shall give a detailed description before Theorem~\ref{theo. semigroup estimate}. We modify the argument in the proofs of \cite[Proposition~2.1]{ARCD2004} and \cite[Theorem~3.1]{MaeTera06}, which treat a semigroup estimate in the uniformly local Lebesgue space framework. In Section~\ref{section. proof of Theorem}, we prove Theorems~\ref{theorem 1} and \ref{theorem 2} by using the fixed point theorem. In Section~\ref{section. Life span}, we apply Theorems~\ref{theorem 1} and \ref{theorem 2} to life span estimates for certain singular initial data, which are typical elements in the uniformly local weak Zygmund spaces. Through these life span estimates, we can understand the correspondence between the local-in-time solvability of \eqref{eq. TFF} and \eqref{eq. Fujita}. Moreover, we introduce a Lee-Ni type estimate which corresponds to the argument in \cite{LeeNi92}.

\section{Semigroup estimate}\label{section. Semigroup estimate}
This section is devoted to the derivation of a semigroup estimate required for the time-fractional problem \eqref{eq. TFF}.

\subsection{Properties of the weak Zygmund space}

For detailed properties of the weak Zygmund type space, see \cite{IokuIshigeKawakami} and references therein. For a Lebesgue measurable function $f$ in $\mathbb{R}^N$, we denote by $\mu_f$ the distribution of $f$,
\[
\mu_f(\lambda):= \left| \left\{ x\in \mathbb{R}^N; |f(x)|>\lambda \right\} \right|\ \ \mbox{for}\ \ \lambda>0.
\]
The non-increasing rearrangement $f^*$ of $f$ is defined by
\[
f^*(s):=\inf\left\{ \lambda>0; \mu_f(\lambda)\le s \right\}\ \ \mbox{for}\ \ s\ge0.
\]
As for the rearrangement, we have the following properties: for $q\ge 1$ and $k\in \mathbb{R}^N$,
\[
\begin{aligned}
\left( kf \right)^{*}&=|k|f^{*},\\
\left( |f|^q \right)^{*}&=\left( f^{*} \right)^q,\\
\int_{\mathbb{R}^N}|f(x)|^{q}\,dx&=\int_{0}^{\infty}f^{*}(s)^{q}ds,\\ f^*(0)&=\|f\|_{L^{\infty}(\mathbb{R}^N)}.
\end{aligned}
\]
Let
\[
f^{**}(s):= \frac{1}{s}\int_{0}^{s} f^{*}(\tau)d\tau.
\]
It is known by \cite[Chapter~2, Proposition~3.3]{BennettSharpley} that,
\[
f^{**}(s)=\frac{1}{s}\sup_{|E|=s}\int_{E}|f(x)|\,dx.
\]
Therefore, by the definition \eqref{eq. definition of Zygmund norm} and the above arguments, we obtain
\begin{equation}\label{eq. Zygmund norm rearrangement form}
\begin{aligned}
\|f\|_{\zygm{q}{\gamma}}&:= \sup_{s>0} \left[ \left( \LOG{s} \right)^{\gamma} \sup_{|E|=s} \int_{E} |f(x)|^q\,dx \right]^{\frac{1}{q}}\\
&= \sup_{s>0} \left[ \left( \LOG{s} \right)^{\gamma} s\left( |f|^q \right)^{**}(s) \right]^{\frac{1}{q}}\\
&=\sup_{s>0} \left[ \left( \LOG{s} \right)^{\gamma} \int_{0}^{s} \left( |f|^{q} \right)^{*}(\tau)d\tau \right]^{\frac{1}{q}}\\
&=\sup_{s>0} \left[ \left( \LOG{s} \right)^{\gamma} \int_{0}^{s} f^{*}(\tau)^qd\tau \right]^{\frac{1}{q}}.
\end{aligned}
\end{equation}
It is known by \cite[Theorem~3.3]{Oneil63} that for $f,\ g\in L^{1}_{\rm loc}$,
\begin{equation}
\left( fg \right)^{**}(s)\le \frac{1}{s}\int_{0}^{s} f^{*}(\tau) g^{*}(\tau) d\tau
\end{equation}
for all $s>0$. Note that $\left( \chi_{E} \right)^{*}(s)= \chi_{[0,|E|]}(s)$ for measurable set $E$ such that $|E|<\infty$. Hence, it follows that
\begin{equation}
\left( f \chi_{E} \right)^{**}\le \frac{1}{s}\int_{0}^{s} f^{*}(\tau) \left( \chi_{E} \right)^{*}(\tau)d\tau = \frac{1}{s} \int_{0}^{\min\left( s, |E| \right)} f^{*}(\tau)d\tau.
\end{equation}
Therefore, combining with the monotonicity of
\[
s\mapsto \LOG{s},
\]
we obtain
\begin{equation}\label{eq. zygmund norm for cut off}
\| f\chi_{E}\|_{\zygm{q}{\gamma}} \le \sup_{0<s<|E|} \left[ \left( \LOG{s} \right)^{\gamma} \int_{0}^{s} f^{*}(\tau)^{q}d\tau \right]^{\frac{1}{q}}.
\end{equation}
Moreover, one has (see \cite[page~9]{IokuIshigeKawakami})
\begin{equation}\label{eq. zygmund norm for cut off2}
\| f \chi_{E}\|_{\zygm{q}{\gamma}} \le \sup_{0<s<|E|}\left[ \left( \LOG{s} \right)^{\gamma}  \int_{0}^{s} \left( f\chi_{E} \right)^{*}(\tau)^q d\tau \right]^{\frac{1}{q}}.
\end{equation}

\subsection{Semigroup estimate}

In what follows, for a set $X$ and maps $a, b :X\to [0,\infty)$, we denote
\[
a(x)\lesssim b(x)
\]
if there exists a positive constant $C>0$ independent of $x$ such that $a(x)\le Cb(x)$ for all $x\in X$. Moreover, we denote
\[
a(x)\simeq b(x)
\]
if there exists $C>0$ such that $C^{-1}b(x)\le a(x)\le Cb(x)$ for all $x\in X$.

We introduce the following important estimates for the proof of the semigroup estimate.  First, recall the following relations on logarithmic functions. For $L>1$ and $k>0$, it follows that
\begin{equation}\label{eq. equivalence of LOG}
\LOG{s}\simeq \log\left( L + \frac{1}{s} \right)\simeq \log\left( e+ \frac{k}{s} \right) \simeq \LOG{s^k},\ \ s>0.
\end{equation}

Moreover, for the estimate of the heat kernel $G_t$, the following is useful:
\begin{equation}\label{eq. estimate of heat kernel by ht}
G_t(x)\lesssim h_t(|x|),
\end{equation}
where
\begin{equation}
h_{t}(r):=t^{-\frac{N}{2}}\left( 1+t^{-\frac{1}{2}}r \right)^{-N-2}.
\end{equation}

Ioku, Ishige, and Kawakami \cite[Proposition~3.2]{IokuIshigeKawakami} constructed the following semigroup estimate in the uniformly local weak Zygmund space:
\begin{equation}\label{eq. semigroup estimate in ul}
\| e^{t\Delta} \varphi\|_{r,\gamma_1;\rho}\le C t^{-\frac{N}{2}\left( \frac{1}{q}-\frac{1}{r} \right)} \left( \LOG{t} \right)^{-\frac{\gamma_2}{q}+\frac{\gamma_1}{r}} \|\varphi\|_{q,\gamma_2;\rho},
\end{equation}
under the restriction of $t\le \rho^{2}$. However, this argument cannot be directly applied to \eqref{eq. TFF}, since we have to use the semigroup estimate within the whole range $t\in (0,\infty)$, even if we just consider the local-in-time solvability, due to the nonlocality with respect to time. Indeed, by Definition~\ref{defi. solution}, we shall estimate
\begin{equation}
\int_{0}^{\infty} h_{\alpha}(\theta) e^{t^{\alpha}\theta \Delta } \varphi\,d\theta\quad\text{and}\quad\int_{0}^{\infty}\theta h_{\alpha}(\theta) e^{(t-s)^{\alpha}\theta\Delta} |u|^{p-1}u(s)\,d\theta ds
\end{equation}
in the proofs of Theorems~\ref{theorem 1} and \ref{theorem 2}. Although the range of time is restricted $0<t<T$, the integral with respect to $\theta$ spans $(0,\infty)$. Therefore, even if we consider the local-in-time solvability of \eqref{eq. TFF}, we need the semigroup estimate which controls any $t\in (0,\infty)$.

\begin{theo}\label{theo. semigroup estimate}
	Suppose that $1\le q\le r\le \infty$ and $\gamma_1 ,\gamma_2 \ge 0$. Assume further that $\gamma_1 \ge \gamma_2$ if $r=q$. Then, there exists $C=C(r,q,\gamma_1,\gamma_2, N)>0$ such that
	\begin{equation}
	\begin{aligned}
	\| e^{t\Delta}\varphi\|_{r,\gamma_1;\rho}\le C&\left[ \rho^{-N\left( \frac{1}{q}-\frac{1}{r} \right) }\left( \LOG{\rho} \right)^{-\frac{\gamma_2}{q}+\frac{\gamma_1}{r}}\right.\\
	&\left.\hspace{2cm}+ t^{-\frac{N}{2}\left(\frac{1}{q}-\frac{1}{r}\right)} \left( \LOG{t} \right)^{-\frac{\gamma_2}{q}+\frac{\gamma_1}{r}} \right]\| \varphi\|_{q,\gamma_2;\rho}
	\end{aligned}
	\end{equation}
	for $\varphi\in \unifz{q}{\gamma_2}$ and for all $t>0$ and $\rho>0$.
\end{theo}

We shall sketch the outline of the proof of Theorem~\ref{theo. semigroup estimate}. In the uniformly local Lebesgue space framework, \cite{ARCD2004, MaeTera06} obtained the semigroup estimate within the whole range $t\in (0,\infty)$. In the following proof of Theorem~\ref{theo. semigroup estimate}, a main idea is similar to that of \cite{ARCD2004,MaeTera06}, the decomposition of $G_t$ and $\varphi$:
\[
e^{t\Delta}\varphi = \sum_{k,k'\in \mathbb{Z}^N} \left( G_t\chi_{S_{k'}(\rho/2)} \right)\ast \left( \varphi \chi_{S_{k}(\rho/2)} \right)
\]
where $S_k(\theta)$ denotes the cube whose center is $\rho k$ for $k\in \mathbb{Z}^N$, sides are of length $\theta$ and parallel to the axis, i.e.,
\[
S_k(\theta):=\left\{ y; \max_{1\le i\le N}|y_i-(\zeta_k)_i|\le \theta \right\},\ \ \zeta_k:=\rho k\quad \text{for}\quad k\in \mathbb{Z}^N.
\]
For $k\in \mathbb{Z}^N$ near the origin, the desired estimate \eqref{eq. semigroup estimate in ul} is obtained by \cite[page 20, (3.19)]{IokuIshigeKawakami}. Thus, it remains to prove the conclusion for $k\in  \mathbb{Z}^N$ far from the origin. By the translation technique, we only have to estimate $\| e^{t\Delta}\varphi\chi_{B(0;\rho)}\|_{\zygm{r}{\gamma_1}}$.
In order to evaluate $\| e^{t\Delta}\varphi\chi_{B(0;\rho)}\|_{\zygm{r}{\gamma_1}}$ by $\| \varphi\|_{q,\gamma_2;\rho}$ from above, we shall make use of the fact that the summation of $\sum_{k}$ can be replaced with the finite one (see \eqref{eq. kkprime finite}), thanks to the decomposition of $G_t$ and $\varphi$.
In this step, although we want to use Young's inequality for the estimate of the convolution $\left( G_t\chi_{S_{k'}(\rho/2)} \right)\ast \left( \varphi \chi_{S_{k}(\rho/2)} \right)$, we unfortunately do not have Young's type inequality associated with the uniformly local weak Zygmund type norm. Therefore, we have to use H\"{o}lder's inequality at the expense of the optimal order.

Finally, it remains to estimate the summation of $\sum_{k'}$ which corresponds to estimates for the heat kernel. In this step, by calculating the integral of the heat kernel within the range excluding the neighborhood of the origin, we compensate for the optimal order. We also explain about this point in Remark~\ref{rem. proof of semigroup} after the proof.

\begin{proof}[Proof of Theorem~\ref{theo. semigroup estimate}]
	By \eqref{eq. semigroup estimate in ul}, the desired estimate is already obtained for $t\le \rho^{2}$. Therefore, it suffices to prove the conclusion provided $t>\rho^2$. Furthermore, by the translation technique, we only have to estimate $\| e^{t\Delta}\varphi\chi_{B(0;\rho)}\|_{\zygm{r}{\gamma_1}}$. As we mentioned above, the cubic decomposition of $\mathbb{R}^N$ plays an essential role in the proof. Let
	\[
	\mathbb{R}^N= \bigcup_{k\in \mathbb{Z}^{N}} S_k(\rho/2)
	\]
	and $u_{k}(t,x):= e^{t\Delta}\left( \varphi \chi_{S_{k}(\rho/2)} \right)(x)$. Then
	\[
	\left| e^{t\Delta}\varphi \right|\le \left| u_0(t,x) \right| + 
	\sum_{\substack{k\in \mathbb{Z}^N, \\ \max_{1\le i\le N}|k_i|> d }} \left| u_{k}(t,x) \right|
	\]
	where $d=d(N)>0$ shall be determined later and
	\[
	u_0(t,x):= \sum_{\substack{m\in \mathbb{Z}^N, \\ \max_{1\le i\le N}|m_i|\le d }} u_m(t,x).
	\]
	\noindent\underline{\textbf{Step 1.}} Estimate for $u_0$ is obtained by the same argument as \cite[page 20, (3.19)]{IokuIshigeKawakami}. Indeed, let $l>0$ be such that
	\[
	\bigcup_{\max_{1\le i\le N}|m_i|\le d } S_{m}(\rho/2) \subset B(0;l \rho).
	\]
	Then, by \cite[Proposition~3.1]{IokuIshigeKawakami},
	\begin{equation}
	\begin{aligned}
	&\| u_0(t)\chi_{B(0;\rho)}\|_{\zygm{r}{\gamma_1}}\\
	&\le \| u_0(t)\|_{\zygm{r}{\gamma_1}}\\
	&\le t^{-\frac{N}{2}\left( \frac{1}{q}-\frac{1}{r} \right)} \left( \LOG{t} \right)^{-\frac{\gamma_2}{q}+\frac{\gamma_1}{r}} \| \varphi \chi_{B(0;l\rho)} \|_{\zygm{q}{\gamma_2}}\\
	&\lesssim t^{-\frac{N}{2}\left( \frac{1}{q}-\frac{1}{r} \right)} \left( \LOG{t} \right)^{-\frac{\gamma_2}{q}+\frac{\gamma_1}{r}} \| \varphi \chi_{B(0;\rho)} \|_{\zygm{q}{\gamma_2}}\\
	&\lesssim t^{-\frac{N}{2}\left( \frac{1}{q}-\frac{1}{r} \right)} \left( \LOG{t} \right)^{-\frac{\gamma_2}{q}+\frac{\gamma_1}{r}} \| \varphi \|_{q,\gamma_2;\rho}.
	\end{aligned}
	\end{equation}
	
	\noindent\underline{\textbf{Step 2.}} From now on, we omit
	\[
	\sum_{k}=\sum_{\substack{k\in \mathbb{Z}^N, \\ \max_{1\le i\le N}|k_i|> d }}
	\]
	for simplicity, since we shall estimate the sum of $u_k$ for $\max_{1\le i\le N} |k_i|>d$. More precisely, we shall deduce that for all $t>0$ and $x\in S_0(\rho/2)$, one has
	\begin{equation}\label{eq. estimat for sum of uk}
		\sum_{k} |u_k(t,x)|\lesssim \left( \LOG{\rho} \right)^{-\frac{\gamma_2}{q}} \rho^{\frac{N}{q^*}} \sum_{|k'|>2\sqrt{N}} h_{t}\left( |\zeta_{k'}|-\frac{\sqrt{N}}{2}\rho \right)\ \|\varphi\|_{q,\gamma_2;\rho}.
	\end{equation}
	It follows by \eqref{eq. estimate of heat kernel by ht} that,
	\[
	\begin{aligned}
	\sum_{k}\left| u_k(t,x) \right|&\lesssim \sum_{k}\int_{\mathbb{R}^N} h_t(x-y)|\varphi(y)\chi_{S_{k}(\rho/2)}(y)|\,dy\\
	&= \sum_{k}\sum_{k'\in\mathbb{Z}^N} \int_{\mathbb{R}^N} h_t(x-y) \chi_{S_{k'}(\rho/2)}(x-y)|\varphi(y)\chi_{S_{k}(\rho/2)}(y)|\,dy.
	\end{aligned}
	\]
	Note that,
	\[
	{\rm supp}\left( h_t \chi_{S_{k'}(\rho/2)} \right)\ast \left( \varphi\chi_{S_k(\rho/2)} \right)\subset S_{k+k'}(\rho),
	\]
	and there exists $M>0$ depending only on $N$ such that
	\begin{equation}\label{eq. kkprime finite}
	S_{k+k'}(\rho)\cap S_0(\rho/2)\neq \emptyset\ \Rightarrow\ |k+k'|\le M.
	\end{equation}
	Indeed, if there exists $y\in S_{k+k'}(\rho)\cap S_{0}(\rho/2)$, then, it holds that
	\[
	\left|y_i\right|\le \frac{\rho}{2}\quad  \text{and}\quad \left| y_i-(k+k')_i\rho \right|\le \rho
	\]
	for all $i=1,2,\cdots,N$. Then, we deduce that $\max_{1\le i\le N}\left| (k+k')_i \right|\le 3/2$.
	
	Hence, we obtain that for all $x\in S_0(\rho/2)$,
	\begin{equation}\label{eq. proof of semigroup 1}
	\sum_{k}\left| u_k(t,x) \right|\lesssim \sum_{|k+k'|\le M} \int_{\mathbb{R}^N} h_t(x-y) \chi_{S_{k'}(\rho/2)}(x-y)|\varphi(y)\chi_{S_{k}(\rho/2)}(y)|\,dy.
	\end{equation}
	Here, we determine $d:= 2\sqrt{N}+M$ so that
	\begin{equation}\label{eq. range of kprime}
	|k'|\ge |k|-M>2\sqrt{N}.
	\end{equation}
	From now on, we omit $S_k:=S_k(\rho/2)$, $\chi_k:=\chi_{S_k(\rho/2)}$, $f_{k'}(y):= h_{t}(x-y)\chi_{k'}(x-y)$ and $g_k(y):=\varphi(y)\chi_k(y)$ for simplicity.
	As we mentioned before starting the proof, we do not have Young's type inequality associated with the uniformly local weak Zygmund type norm. Thus, we use H\"{o}lder's inequality with respect to $y$ to obtain
	\begin{equation}
	\begin{aligned}
	&\int_{\mathbb{R}^N} h_t(x-y) \chi_{k'}(x-y)|\varphi(y)\chi_{k}(y)|\,dy\\
	&=\int_{\mathbb{R}^N}|f_{k'}(y)g_k(y)|\,dy\\
	&\le \sup_{s>0} \left[ \left( \LOG{s} \right)^{-\frac{\gamma_2}{q}} \left( \LOG{s} \right)^{\frac{\gamma_2}{q}} \sup_{|E|=s} \int_{E}|f_{k'}(y)g_k(y)|\,dy \right]\\
	&\le \sup_{s>0} \left[ \left( \LOG{s} \right)^{-\frac{\gamma_2}{q}} \left( \LOG{s} \right)^{\frac{\gamma_2}{q}}\right.\\
	&\hspace{3cm}\left.\times\sup_{|E|=s} \left( \int_{E} |f_{k'}(y)|^{q^*}\,dy \right)^{\frac{1}{q^*}} \left( \int_{E} |g_k(y)|^{q} \,dy\right)^{\frac{1}{q}} \right]\\
	&\le \sup_{s>0} \left[ \left( \LOG{s} \right)^{-\frac{\gamma_2}{q}} \sup_{|E|=s} \left( \int_{E} |f_{k'}(y)|^{q^*}\,dy \right)^{\frac{1}{q^*}}\right]\\
	&\hspace{3cm} \times \sup_{k} \sup_{s>0}\left[ \left( \LOG{s} \right)^{\frac{\gamma_2}{q}} \left( \int_{E} |g_k(y)|^{q} \,dy\right)^{\frac{1}{q}} \right]\\
	&\le \sup_{s>0} \left[ \left( \LOG{s} \right)^{-\frac{\gamma_2}{q}} \sup_{|E|=s} \left( \int_{E} |f_{k'}(y)|^{q^*}\,dy \right)^{\frac{1}{q^*}}\right] \| \varphi\|_{q,\gamma_2;\rho}.
	\end{aligned}
	\end{equation}
	By \eqref{eq. kkprime finite}, it follows that $k\in B(-k';M)\cap \mathbb{Z}^N$. In particular, the summation $\sum_{k}$ must be finite for each $k'$, depending only on $N$. Therefore, it is verified by \eqref{eq. proof of semigroup 1} and \eqref{eq. range of kprime} that
	\begin{equation}\label{eq. proof of semigroup 2}
	\sum_{k} |u_{k}(t,x)| \lesssim \sum_{|k'|>2\sqrt{N}}\sup_{s>0}\left[ \left( \LOG{s} \right)^{-\frac{\gamma_2}{q}} \sup_{|E|=s}\left( \int_{E} |f_{k'}(y)|^{q^*}\,dy \right)^{\frac{1}{q^*}} \right] \|\varphi\|_{q,\gamma_2;\rho}.
	\end{equation}
	We further estimate the summation with respect to $k'>2\sqrt{N}$ in \eqref{eq. proof of semigroup 2}. Since $S_{k'}(\rho/2) \subset B(\zeta_{k'};\sqrt{N}\rho/2)=:B_{k'}$, we obtain by \eqref{eq. Zygmund norm rearrangement form} and \eqref{eq. zygmund norm for cut off2} that,
	\begin{equation}\label{eq. proof of semigroup 3}
	\begin{aligned}
	&\sup_{s>0} \left[ \left( \LOG{s} \right)^{-\frac{\gamma_2}{q}} \sup_{|E|=s} \left( \int_{E} |f_{k'}(y)|^{q^*}\,dy \right)^{\frac{1}{q^*}} \right]\\
	&\le \sup_{s>0} \left[ \left( \LOG{s} \right)^{-\frac{\gamma_2 q^*}{q}} \sup_{|E|=s} \int_{E} \left| h_t(y)\chi_{k'}(y) \right|^{q*}\,dy \right]^{\frac{1}{q*}}\\
	& \le \sup_{s>0} \left[ \left( \LOG{s} \right)^{-\frac{\gamma_2 q^*}{q}} \sup_{|E|=s} \int_{E} \left| h_t(y)\chi_{B_{k'}}(y) \right|^{q*}\,dy \right]^{\frac{1}{q*}}\\
	&\le \sup_{0<s<c\rho^N} \left[ \left( \LOG{s} \right)^{-\frac{\gamma_2 q^*}{q}} \int_{0}^{\min (s, c\rho^N)}\left|  \left( h_t \chi_{B_{k'}} \right)^{*}(\tau) \right|^{q^*}d\tau \right]^{\frac{1}{q^*}}\\
	&:=J_{k'}.
	\end{aligned}
	\end{equation}
	Here,
	\[
	c:= \omega_N \left( \frac{\sqrt{N}}{2} \right)^{N}
	\]
	and $\omega_N$ denotes the volume of the $N$-dimensional unit ball. By the monotonicity of the rearrangement and $\left( h_t\chi_{B_{k'}} \right)^*(0)=h_t\left( |\zeta_{k'}|-\sqrt{N}\rho/2 \right)$, it follows that
	\begin{equation}\label{eq. proof of semigroup 4}
	\begin{aligned}
	J_{k'}&=\sup_{0<s<c\rho^N} \left[ \left( \LOG{s} \right)^{-\frac{\gamma_2 q^*}{q}} \int_{0}^{s} \left| \left( h_t\chi_{B_{k'}} \right)^{*}(\tau) \right|^{q^*}d\tau \right]^{\frac{1}{q^*}}\\
	&\le\sup_{0<s<c\rho^N}\left[ \left( \LOG{s} \right)^{-\frac{\gamma_2 q^*}{q}} sh_t\left(|\zeta_k'|-\frac{\sqrt{N}}{2}\rho\right)^{q^*} \right]^{\frac{1}{q^*}}\\
	&\lesssim \left( \LOG{\rho} \right)^{-\frac{\gamma_2}{q}} \rho^{\frac{N}{q^*}} h_t \left( |\zeta_{k'}| -\frac{\sqrt{N}}{2}\rho \right),
	\end{aligned}
	\end{equation}
	since
	\[
	s\mapsto \left( \LOG{s} \right)^{-\frac{\gamma_2}{q}}s
	\]
	is monotone increasing. Thus, by \eqref{eq. proof of semigroup 1}, \eqref{eq. proof of semigroup 2}, \eqref{eq. proof of semigroup 3}, and \eqref{eq. proof of semigroup 4}, we obtain the desired estimate \eqref{eq. estimat for sum of uk}.
	
	Thus, in order to derive the conclusion, it remains to deal with
	\begin{equation}\label{eq. proof of semigroup 8}
		\sum_{|k'|>2\sqrt{N}} h_t\left( |\zeta_{k'}|-\frac{\sqrt{N}}{2}\rho \right)
	\end{equation}
	which corresponds to the estimate of the heat kernel far from the origin.
	
	\noindent\underline{\textbf{Step 3.}} We shall estimate \eqref{eq. proof of semigroup 8}. Note that for $y\in B_{k'}$, the condition $|y-\zeta_{k'}|\le\sqrt{N}\rho/2$ implies $|y|\le \sqrt{N}\rho/2 + |\zeta_{k'}|$. Since $|\zeta_{k'}|> 2\sqrt{N}\rho>\sqrt{N}\rho/2$ for $|k'|>2\sqrt{N}$, we have
	\[
	|\zeta_{k'}|= \frac{1}{2}\left( |\zeta_{k'}|+|\zeta_{k'}| \right)\ge \frac{1}{2}\left( |\zeta_{k'}| +\frac{\sqrt{N}}{2}\rho \right)\ge  \frac{1}{2}|y|\quad \text{for}\quad y\in B_{k'}.
	\]
	In particular, for $y\in B_{k'}$,
	\[
	1+t^{-\frac{1}{2}}\left( \frac{1}{2}|y|-\frac{\sqrt{N}}{2}\rho \right)\le 1+t^{-\frac{1}{2}}\left( |\zeta_{k'}|-\frac{\sqrt{N}}{2}\rho \right).
	\]
	Then, for  $y\in B_{k'}$, it follows that
	\[
	\begin{aligned}
	h_t\left( |\zeta_{k'}|-\frac{\sqrt{N}}{2}\rho \right)&= t^{-\frac{N}{2}}\left( 1+ t^{-\frac{1}{2}} \left( |\zeta_{k'}|-\frac{\sqrt{N}}{2}\rho \right)  \right)^{-N-2}\\
	&\le t^{-\frac{N}{2}} \left( 1+t^{-\frac{1}{2}} \left( \frac{1}{2}|y|-\frac{\sqrt{N}}{2}\rho \right)  \right)^{-N-2}.
	\end{aligned}
	\]
	Note that
	\[
	|k'|>2\sqrt{N}\ \Rightarrow\ |\zeta_{k'}|>2\sqrt{N} \rho\ \Rightarrow\ |y|>\sqrt{N}\rho\ \ \mbox{for}\ \ y\in B_{k'}.
	\]
	Therefore,
	\begin{equation}\label{eq. proof of semigroup 5}
	\begin{aligned}
	&\sum_{|k'|>2\sqrt{N}} h_t \left( |\zeta_{k'}|-\frac{\sqrt{N}}{2}\rho \right)\\
	&\le t^{-\frac{N}{2}} \sum_{|k'|>2\sqrt{N}} \frac{1}{|B_{k'}|} \int_{B_{k'}}\left( 1+t^{-\frac{1}{2}} \left( \frac{1}{2}|y|-\frac{\sqrt{N}}{2}\rho \right)  \right)^{-N-2}\,dy\\
	&\lesssim t^{-\frac{N}{2}}\rho^{-N} \int_{\mathbb{R}^N\setminus B(0;\sqrt{N}\rho) } \left( 1+t^{-\frac{1}{2}} \left( \frac{1}{2}|y|-\frac{\sqrt{N}}{2}\rho \right)  \right)^{-N-2}\,dy\\
	&\lesssim t^{-\frac{N}{2}}\rho^{-N}\int_{\sqrt{N}\rho}^{\infty} \left( 1+ t^{-\frac{1}{2}} \left( \frac{1}{2}r-\frac{\sqrt{N}}{2}\rho \right) \right)^{-N-2} r^{N-1}dr.
	\end{aligned}
	\end{equation}
	We divide the cases with respect to the dimension $N\ge 1$ in order to estimate \eqref{eq. proof of semigroup 5}. If $1\le N\le 4$, it follows that $1- t^{-\frac{1}{2}}\sqrt{N}\rho/2 \ge 0$ by the assumption $t^{\frac{1}{2}}\ge \rho$. Therefore,
	\begin{equation}
	\begin{aligned}
	&\int_{\sqrt{N}\rho}^{\infty} \left( 1+ t^{-\frac{1}{2}} \left( \frac{1}{2}r-\frac{\sqrt{N}}{2}\rho \right) \right)^{-N-2} r^{N-1}dr\\
	&\simeq t^{\frac{N}{2}-\frac{1}{2}}\int_{\sqrt{N}\rho}^{\infty} \left( 1+ t^{-\frac{1}{2}} \left( \frac{1}{2}r-\frac{\sqrt{N}}{2}\rho \right) \right)^{-N-2} \left( \frac{1}{2}t^{-\frac{1}{2}} r \right)^{N-1} dr \\
	&\lesssim t^{\frac{N}{2}-\frac{1}{2}} \int_{\sqrt{N}\rho}^{\infty} \left( 1+ t^{-\frac{1}{2}} \left( \frac{1}{2}r-\frac{\sqrt{N}}{2}\rho \right) \right)^{-N-2} \left(1+ \frac{1}{2}t^{-\frac{1}{2}} r -\frac{\sqrt{N}}{2}t^{-\frac{1}{2}}\rho \right)^{N-1} dr\\
	&\simeq t^{\frac{N}{2}-\frac{1}{2}}  \int_{\sqrt{N}\rho}^{\infty} \left( 1+ t^{-\frac{1}{2}} \left( \frac{1}{2}r-\frac{\sqrt{N}}{2}\rho \right) \right)^{-3} dr\\
	&\lesssim t^{\frac{N}{2}}.
	\end{aligned}
	\end{equation}
	Otherwize, if $N>4$, it follows that
	\[
	\begin{aligned}
	1+t^{-\frac{1}{2}} \left( \frac{1}{2} r -\frac{\sqrt{N}}{2} \rho \right)&= 1+ t^{-\frac{1}{2}} \frac{\sqrt{N}}{2} \left( \frac{1}{\sqrt{N}}r-\rho \right)\\
	&\ge 1 +t^{-\frac{1}{2}}\left( \frac{1}{\sqrt{N}}r-\rho \right).
	\end{aligned}
	\]
	Thus, by the assumption $t^{\frac{1}{2}}\ge \rho$, we get
	\begin{equation}
	\begin{aligned}
	&\int_{\sqrt{N}\rho}^{\infty} \left( 1+ t^{-\frac{1}{2}} \left( \frac{1}{2}r-\frac{\sqrt{N}}{2}\rho \right) \right)^{-N-2} r^{N-1}dr\\
	&\lesssim \int_{\sqrt{N}\rho}^{\infty} \left( 1+ t^{-\frac{1}{2}} \left( \frac{1}{\sqrt{N}}r-\rho \right) \right)^{-N-2} r^{N-1}dr\\
	&\simeq t^{\frac{N}{2}-\frac{1}{2}} \int_{\sqrt{N}\rho}^{\infty} \left( 1+ t^{-\frac{1}{2}} \left( \frac{1}{\sqrt{N}}r-\rho \right) \right)^{-N-2} \left( \frac{1}{\sqrt{N}} t^{-\frac{1}{2}} r \right)^{N-1}dr\\
	&\lesssim t^{\frac{N}{2}-\frac{1}{2}} \int_{\sqrt{N}\rho}^{\infty} \left( 1+ t^{-\frac{1}{2}} \left( \frac{1}{\sqrt{N}}r-\rho \right) \right)^{-N-2} \left(1+ \frac{1}{\sqrt{N}} t^{-\frac{1}{2}} r -t^{-\frac{1}{2}}\rho \right)^{N-1}dr\\
	&\lesssim t^{\frac{N}{2}}.
	\end{aligned}
	\end{equation}
	Therefore, by \eqref{eq. estimat for sum of uk} and \eqref{eq. proof of semigroup 5}, for any $x\in S_0(\rho/2)$, we obtain
	\begin{equation}\label{eq. proof of semigroup 6}
	\sum_{k} |u_k(t,x)|\lesssim \left( \LOG{\rho} \right)^{-\frac{\gamma_2}{q}} \rho^{-\frac{N}{q}} \| \varphi\|_{q,\gamma_2;\rho}.
	\end{equation}
	
	\noindent\underline{\textbf{Step 4.}} We shall deduce the desired conclusion. Since $(\chi_E)^*=\chi_{[0,|E|]}$, it follows that
	\begin{equation}\label{eq. proof of semigroup 7}
	\begin{aligned}
	\| \chi_{S_0(\rho/2)}\|_{\zygm{r}{\gamma_1}}&= \sup_{s>0}\left[ \left( \LOG{s} \right)^{\gamma_1} \int_{0}^{s} \chi_{[0,\rho^N]}(\tau)d\tau \right]^{\frac{1}{r}}\\
	&=\sup_{s>0}\left[ \left( \LOG{s} \right)^{\gamma_1} \int_{0}^{\min\left( s, \rho^N\right)}d\tau \right]^{\frac{1}{r}}\\
	&= \sup_{0<s<\rho^N} \left[ \left( \LOG{s} \right)^{\gamma_1} s \right]^{\frac{1}{r}}\\
	&\lesssim \left( \LOG{\rho} \right)^{\frac{\gamma_1}{r}}\rho^{\frac{N}{r}}.
	\end{aligned}
	\end{equation}
	Indeed, use \eqref{eq. equivalence of LOG} and take $L> 1$ so that
	\[
	s\mapsto s \left( \log\left( L +\frac{1}{s} \right) \right)^{\gamma_1}
	\]
	is monotone increasing. Therefore, combining with \eqref{eq. proof of semigroup 6} and \eqref{eq. proof of semigroup 7}, we deduce
	\[
	\begin{aligned}
	&\left\| \sum_{k} u_{k}(t) \chi_{S_0(\rho/2)} \right\|_{\zygm{r}{\gamma_1}}\\
	&\lesssim \left( \LOG{\rho} \right)^{-\frac{\gamma_2}{q}} \rho^{-\frac{N}{q}} \|\chi_{S_0(\rho/2)} \|_{\zygm{r}{\gamma_1}} \|\varphi\|_{q,\gamma_2;\rho}\\
	&\lesssim \left( \LOG{\rho} \right)^{-\frac{\gamma_2}{q}+\frac{\gamma_1}{r}}\rho^{-N\left( \frac{1}{q}-\frac{1}{r} \right)}\|\varphi\|_{q,\gamma_2;\rho}.
	\end{aligned}
	\]
	Note that for a function $f$, it follows that
	\begin{equation}
	\| f_{\chi_{S_0(\rho/2)}}\|_{\zygm{r}{\gamma_1}}\simeq \| f_{\chi_{B(0;\rho)} } \|_{\zygm{r}{\gamma_1}}.
	\end{equation}
	
\end{proof}

\begin{rem}\label{rem. proof of semigroup}
	{\rm
		We compare our method with the arguments in the proof of \cite[Proposition~2.1]{ARCD2004} or \cite[Theorem~3.1]{MaeTera06}. They use Young's inequality for the estimate related to $ \|G_t\ast \varphi\|_{L^{p}_{\rm ul}}$, whereas we cannot obtain Young's type inequality in terms of the norm $\| \cdot \|_{r,\gamma_1 ; \rho}$. Therefore, instead of Young's inequality, we use H\"{o}lder's inequality in \eqref{eq. proof of semigroup 2}, at the expense of the optimal order, to deduce the $L^{\infty}$ estimate of $\sum_{k} u_k(t,x)$. However, we can take advantage since we calculate the integral of the term corresponding to $G_t S_{{k'}}$ in \eqref{eq. proof of semigroup 5}, within the range excluding the neighborhood of the origin, although the previous studies calculate the integral of $G_t$ in the whole space, at the corresponding step of the proof. Thanks to this calculation, we can compensate for the loss due to the use of H\"{o}lder's inequality, and deduce the desired estimate.
		
		Furthermore, in \cite{MaeTera06}, they prove the estimate for $\rho=1$ without the loss of generality, by using the scaling argument and the homogeneity of the desired inequality. On the other hand, we need to prove Theorem~\ref{theo. semigroup estimate} for general $\rho>0$ directly, since the logarithmic function does not have the homogeneity.
	}
\end{rem}

\section{Proofs of Theorems~\ref{theorem 1} and \ref{theorem 2}}\label{section. proof of Theorem}

This section is devoted to the proofs of Theorems~\ref{theorem 1} and \ref{theorem 2}. Firstly, we prepare some lemmata, that are mainly required in order to deal with some logarithmic type functions.

\subsection{Lemmata}

We shall use the following lemma in the proofs of Theorems~\ref{theorem 1} and \ref{theorem 2}, when we estimate the nonlinear term. In particular, the $\alpha$-dependence of the right hand sides of \eqref{eq. main vanish} and \eqref{eq. main notvanish} derives from the following estimates. 

\begin{lem}\label{lem. log integral} The following hold.
	\begin{enumerate}
		\item[{\rm (i)}] For $0<a\le 1$ and $q>1$, there exists $C=C(q)>0$ such that
		\begin{equation}\label{eq. log integral not vanish}
			\int_{0}^{s} \tau^{-a} \left( \LOG{\tau} \right)^{-q}d\tau \le C \log\left( e+2s \right) s^{1-a}\left( \LOG{s} \right)^{-q+1}
		\end{equation}
		for all $s>0$.
		
		\item[{\rm (ii)}] For $0<a<1$ and $q\in \mathbb{R}$, there exists $C=C(a, q)>0$ such that
		\begin{equation}\label{eq. log integral vanish}
			\int_{0}^{s}\tau^{-a} \left( \LOG{\tau} \right)^{-q}
			d\tau\le C (1-a)^{-1} s^{1-a}\left( \LOG{s} \right)^{-q}
		\end{equation}
		for all $s>0$. In particular, we can take $C>0$ independently of $a\in (0,1)$ when $q\ge 0$.
	\end{enumerate}
\end{lem}

\begin{proof}
	(i) When $0<s\le 1/2$, it follows that
	\[
	\begin{aligned}
	\int_{0}^{s} \tau^{-a} \left( \LOG{\tau} \right)^{-q}d\tau&\lesssim \int_{0}^{s} \tau^{-a} \left| \log \tau \right| d\tau\\
	&\simeq \int_{-\log s}^{\infty} e^{-(1-a)w}w^{-q}dw\\
	&\lesssim s^{1-a}\frac{1}{q-1}\left| \log s \right|^{-q+1}\\
	&\lesssim s^{1-a} \left( \LOG{s} \right)^{-q+1}.
	\end{aligned}
	\]
	On the other hand, when $1/2<s$,
	\[
	\begin{aligned}
	\int_{0}^{1/2}\tau^{-a}\left( \LOG{\tau} \right)^{-q}d\tau &\lesssim \left( 2^{-1} \right)^{1-a}\left( \LOG{2^{-1}} \right)^{-q+1}\\
	&\lesssim s^{1-a} \left( \LOG{s} \right)^{-q+1}\\
	\end{aligned}
	\]
	due to the monotonicity. Moreover,
	\[
	\begin{aligned}
	\int_{1/2}^{s}\tau^{-a}\left( \LOG{\tau} \right)^{-q}d\tau&\le \left( \LOG{s} \right)^{-q} \int_{1/2}^{s} \tau^{-a} \tau^{1-1}d\tau\\
	&\le s^{1-a}\left( \LOG{s} \right)^{-q} \left( \log s-\log 1/2 \right)\\
	&\le \log\left( e+2s \right)s^{1-a}\left( \LOG{s} \right)^{-q+1}.
	\end{aligned}
	\]
	
	The argument (ii) is trivial when $q\ge 0$, since 
	\[
	s \mapsto \left( \LOG{s} \right)^{-q}
	\]
	is non-decreasing. Suppose that $q<0$ and set $\delta= (1-a)/2$. Take $L>1$ so that
	\[
	\tau \mapsto \tau^{\delta} \left( \log\left( L + \frac{1}{\tau} \right) \right)^{-q}
	\]
	is non-decreasing. Then,
	\[
	\begin{aligned}
	\int_{0}^{s} \tau^{-a} \left( \LOG{\tau} \right)^{-q} d\tau &\lesssim s^{\delta}\left( \log\left( L + \frac{1}{s} \right) \right)^{-q} \int_{0}^{s} \tau^{-a-\delta}d\tau\\
	&\lesssim (1-a)^{-1} s^{1-a}\left( \LOG{s} \right)^{-q}.
	\end{aligned}
	\]
\end{proof}

\begin{rem}
	{\rm
		We mainly use Lemma~\ref{lem. log integral} with $a=\alpha$ in the proofs of Theorems~\ref{theorem 1} and \ref{theorem 2}. Note that the right hand side of \eqref{eq. log integral not vanish} stays finite even if $a\nearrow 1$, whereas the right hand side of \eqref{eq. log integral vanish} diverges when $a\nearrow 1$. These differences inherit whether the decay of the logarithmic function at $\tau=0$ compensate for the integrability of $\tau^{-1}$ or not. In the proof of Theorem~\ref{theorem 1}, we use \eqref{eq. log integral vanish} in the treatment of the nonlinear term, since we just assume that $\gamma\ge 0$. On the other hand, we can use \eqref{eq. log integral not vanish}, under the further assumption $\gamma=N/2$ in Theorem~\ref{theorem 2}. Thus, we can obtain the argument which does not vanish as $\alpha \nearrow 1$.
	}
\end{rem}

The following lemma shall be used in the proof of Theorems~\ref{theorem 1} and \ref{theorem 2} in order to deal with the nonlocality due to the integral with respect to $\theta\in(0,\infty)$ in \eqref{eq. defi of P and S}.
\begin{lem}\label{lem. log constant}
	For $0<k<1$, the following hold.
	\begin{enumerate}
		\item[{\rm (i)}] It follows that
		\[
		k\LOG{s}\le\log\left( e+\frac{k}{s} \right) \le \frac{1}{k} \LOG{s}
		\]	
		for all $s>0$.
		
		\item[{\rm (ii)}] For all $\epsilon >0$, there exists $C>0$ such that
		\[
		\LOG{sk}\le C k^{-\epsilon} \LOG{s}
		\]
		for all $s>0$.
	\end{enumerate}
\end{lem}

\begin{proof}
	(i) Let
	\[
	f(s):= \frac{1}{k}\log\left( e+\frac{k}{s} \right)-\LOG{s}.
	\]
	We easily check the lower estimate by
	\[
	f'(s)=s^{-1}\left( \frac{1}{es+1}-\frac{1}{es+k} \right)\le 0
	\]
	and
	\[
	\lim_{s\to \infty}f(s)=\frac{1}{k}-1>0.
	\]
	Conversely, set
	\[
	g(s):=\frac{1}{k} \LOG{s}-\log\left( e+ \frac{k}{s} \right).
	\]
	Then,
	\[
	g'(s)= \frac{s^{-1}}{(es+k)(es+1)} \left[ es\left( k-\frac{1}{k} \right) + (k-1) \right]\le 0
	\]
	and
	\[
	\lim_{s\to \infty} g(s)=\frac{1}{k}-1>0
	\]
	implies the upper estimate.
	
	(ii) For any $s>0$ and $0<k<1$, we have
	\[
	\begin{aligned}
	\LOG{sk}&\le \log\left( e+\frac{1}{s} \right)\left( e+\frac{1}{k} \right)\\
	&\le \LOG{s}+\LOG{k}\\
	&\le 2\LOG{k}\LOG{s}.
	\end{aligned}
	\]
	The conclusion is easily obtained since 
	\[
	(0,1)\ni k \to k^{\epsilon}\LOG{k}
	\]
	is bounded.
	
\end{proof}

The following is H\"{o}lder's inequality in the uniformly local weak Zygmund type space framework.
\begin{lem}[See {\cite[Lemma 2.2]{IokuIshigeKawakami}}]
	\label{lem. Holder}
	Let $1\le q_1,\ q_2\le \infty$ and $\gamma_1,\ \gamma_2\ge 0$ be such that
	\[
	1=\frac{1}{q_1}+\frac{1}{q_2},\ \mbox{and}\ \gamma=\frac{\gamma_1}{q_1}+\frac{\gamma_2}{q_2}.
	\]
	Then, for any $f_1\in \unifz{q_1}{\gamma_1}$, $f_2\in \unifz{q_2}{\gamma_2}$ and $\rho>0$, it follows that
	\begin{equation}
	\| f_1 f_2\|_{1,\gamma ;\rho}\le \|f_1\|_{q_1,\gamma_1;\rho} \|f_2\|_{q_2,\gamma_2;\rho}.
	\end{equation}
\end{lem}

The following is needed when we deal with the nonlinear term $|u|^{p-1}u$.
\begin{lem}[See {\cite[Lemma 2.3]{IokuIshigeKawakami}}]
	\label{lem. power and norm}
	For any $1\le q<\infty$, $\gamma\ge 0$, and $r>0$ with $rq\ge 1$, it follows that
	\[
	\| |f|^{r} \|_{q,\gamma;\rho}= \| f\|_{rq.\gamma;\rho}^{r}
	\]
	for all $f\in \unifz{q}{\gamma}$ and $\rho>0$.
\end{lem}

\subsection{Proof of Theorem~\ref{theorem 1}}

\begin{proof}[Proof of Theorem~\ref{theorem 1}]
	Let $0\le \tilde{\gamma} <\gamma$ be fixed arbitrarily. We define a function space as follows:
	\[
	\begin{aligned}
	\mathcal{X}_T&:= \left\{ u\in L^{\infty}_{\rm loc}(0,T; \unifz{p}{\tilde{\gamma}}) ; \|u\|_{\mathcal{X}_T}<\infty\right\},\\
	\|u\|_{\mathcal{X}_T}&:= \sup_{0<t<T} t^{\frac{\alpha N}{2}\left( 1-\frac{1}{p} \right)}\left( \LOG{t} \right)^{\gamma - \frac{\tilde{\gamma}}{p}} \|u(t)\|_{p,\tilde{\gamma}; T^{\alpha/2}},
	\end{aligned}
	\]
	and its closed ball
	\[
	\mathcal{B}_{K,T}:= \left\{ u\in \mathcal{X}_T; \|u\|_{\mathcal{X}_T}\le K \right\}.
	\]
	Let
	\[
	\Phi(u)(t):= P_\alpha(t)\varphi+\alpha\int_0^t (t-s)^{\alpha-1}S_\alpha(t-s)|u(s)|^{p-1}u(s) \,ds.
	\]
	We use the fixed point theorem for $\Phi$ in $\mathcal{B}_{K,T}$ for appropriate $K>0$.
	
	\noindent\underline{\textbf{Step 1.}} We show that $\Phi(\mathcal{B}_{K,T})\subset \mathcal{B}_{K,T}$.
	
	Firstly, by Theorem~\ref{theo. semigroup estimate}, we have
	\[
	\begin{aligned}
	&\left\| P_{\alpha}(t)\varphi \right\|_{p,\tilde{\gamma};T^{\alpha/2}}\\
	&\le \int_{0}^{\infty} h_{\alpha}(\theta) \| e^{t^{\alpha}\theta \Delta} \varphi\|_{p,\tilde{\gamma};T^{\alpha/2}}d\theta\\
	&\lesssim \int_{0}^{\infty}h_{\alpha}(\theta) \left[ T^{-\frac{\alpha N}{2}\left( 1-\frac{1}{p} \right)}\left( \LOG{T^{\alpha/2}} \right)^{\frac{\tilde{\gamma}}{p}-\gamma} \right.\\
	&\hspace{1cm}\left.+\ t^{-\frac{\alpha N}{2}\left(1-\frac{1}{p} \right)} \left( \LOG{t^{\alpha}\theta} \right)^{\frac{\tilde{\gamma}}{p}-\gamma} \theta^{-\frac{N}{2}\left( 1-\frac{1}{p} \right)} \right] d\theta \|\varphi\|_{1,\gamma;T^{\alpha/2}}\\
	&\lesssim t^{-\frac{\alpha N}{2}\left( 1-\frac{1}{p} \right)} \left( \LOG{t} \right)^{\frac{\tilde{\gamma}}{p}-\gamma }\| \varphi \|_{1,\gamma;T^{\alpha/2}}\\
	&\hspace{1cm}+ t^{-\frac{\alpha N}{2}\left( 1-\frac{1}{p} \right)} \int_{0}^{\infty} h_{\alpha}(\theta) \left(\LOG{t^{\alpha}\theta}\right)^{\frac{\tilde{\gamma}}{p}-\gamma} \theta^{-\frac{N}{2}\left(1-\frac{1}{p}\right)}d\theta   \| \varphi \|_{1,\gamma;T^{\alpha/2}}.
	\end{aligned}
	\]
	Indeed, for the first term, if necessary we use \eqref{eq. equivalence of LOG} and take $L>1$ so that
	\[
	s\mapsto s^{-\frac{\alpha N}{2}\left( 1-\frac{1}{p} \right)}\left( \log\left( L + \frac{1}{s} \right) \right)^{\frac{\tilde{\gamma}}{p}-\gamma}
	\]
	is monotone decreasing. For the estimate for the second term, we use Lemma~\ref{lem. log constant} to deal with the integral with respect to $\theta\in (0,\infty)$:
	
	If $0<\theta<1\ \Leftrightarrow\ 1<\frac{1}{\theta}<\infty$, then by the monotonicity,
	\[
	\LOG{t^{\alpha}\theta}\ge \LOG{t^{\alpha}}.
	\]
	
	If $1\le \theta <\infty\ \Leftrightarrow\ 0<\frac{1}{\theta}\le 1$, by Lemma~\ref{lem. log constant} (i),
	\[
	\LOG{t^{\alpha}\theta}\ge \frac{1}{\theta} \LOG{t^{\alpha}}.
	\]
	By $\frac{\tilde{\gamma}}{p}-\gamma<0$, $-\frac{N}{2}(1-\frac{1}{p})>-1$ and \eqref{eq. integral of halpha}, we obtain
	\[
	\begin{aligned}
	&\left( \int_{0}^{1}+\int_{1}^{\infty} \right) h_{\alpha}(\theta) \left(\LOG{t^{\alpha}\theta}\right)^{\frac{\tilde{\gamma}}{p}-\gamma} \theta^{-\frac{N}{2}\left(1-\frac{1}{p}\right)}d\theta\\
	&\le \left[ \int_{0}^{1} h_{\alpha}(\theta) \theta^{-\frac{N}{2}\left(1-\frac{1}{p} \right)}d\theta + \int_{1}^{\infty} h_{\alpha}(\theta)\theta^{-\frac{N}{2}\left(1-\frac{1}{p} \right) +\gamma -\frac{\tilde{\gamma}}{p} } d\theta \right] \left( \LOG{t} \right)^{\frac{\tilde{\gamma}}{p}-\gamma}\\
	&\le \left[ \int_{0}^{\infty} h_{\alpha}\theta^{-\frac{N}{2}\left( 1-\frac{1}{p} \right)} d\theta + \int_{0}^{\infty} h_{\alpha}(\theta) \theta^{\gamma-\frac{\tilde{\gamma}}{p}} d\theta \right]\left( \LOG{t} \right)^{\frac{\tilde{\gamma}}{p}-\gamma}\\
	&\le C \left( \LOG{t} \right)^{\frac{\tilde{\gamma}}{p}-\gamma}
	\end{aligned}
	\]
	where
	\[
	C= \frac{\Gamma\left( 1-\frac{N}{2}\left( 1-\frac{1}{p} \right) \right)}{\Gamma\left( 1-\frac{\alpha N}{2}\left( 1-\frac{1}{p} \right) \right)} + \frac{\Gamma \left( 1+\gamma-\frac{\tilde{\gamma}}{p} \right)}{ \Gamma\left( 1+ \alpha \left( \gamma -\frac{\tilde{\gamma}}{p} \right) \right) },
	\]
	which is finite independently of $\alpha \in (0,1)$. Therefore,
	\[
	\| P_{\alpha}(\cdot)\varphi\|_{\mathcal{X}_T}\le C\|\varphi\|_{1,\gamma;T^{\alpha/2}}.
	\]
	
	Next, we estimate the nonlinear term. That is, for $u\in \mathcal{B}_{K,T}$, we shall calculate
	\[
	\begin{aligned}
	\int_{0}^{t} (t-s)^{\alpha-1}\int_{0}^{\infty} \theta h_{\alpha}(\theta) \|e^{(t-s)^{\alpha}\theta \Delta} |u(s)|^{p-1}u(s)\|_{p,\tilde{\gamma};T^{\alpha/2}}d\theta ds.
	\end{aligned}
	\]
	By Theorem~\ref{theo. semigroup estimate} and Lemma~\ref{lem. power and norm},
	\[
	\begin{aligned}
	&\left\| e^{(t-s)^{\alpha}\theta \Delta} |u(s)|^{p-1}u(s) \right\|_{p,\tilde{\gamma};T^{\alpha/2}}\\
	&\lesssim \left[ T^{-\frac{\alpha N}{2}\left(1-\frac{1}{p}\right)} \left( \LOG{T} \right)^{\frac{\tilde{\gamma}}{p}-\tilde{\gamma}}\right.\\
	&\left.\hspace{1cm} + \left( (t-s)^{\alpha}\theta \right)^{-\frac{\alpha N}{2}\left( 1-\frac{1}{p} \right)} \left( \LOG{(t-s)^{\alpha}\theta} \right)^{\frac{\tilde{\gamma}}{p}-\tilde{\gamma}} \right] \| |u(s)|^p\|_{1,\tilde{\gamma};T^{\alpha/2}}\\
	&\lesssim \left[ t^{-\frac{\alpha N}{2}\left(1-\frac{1}{p}\right)} \left( \LOG{t} \right)^{\frac{\tilde{\gamma}}{p}-\tilde{\gamma}}\right.\\
	&\left.\hspace{1cm} + \left( (t-s)^{\alpha}\theta \right)^{-\frac{\alpha N}{2}\left( 1-\frac{1}{p} \right)} \left( \LOG{(t-s)^{\alpha}\theta} \right)^{\frac{\tilde{\gamma}}{p}-\tilde{\gamma}} \right] \| u(s)\|^{p}_{p,\tilde{\gamma};T^{\alpha/2}}.
	\end{aligned}
	\]
	Similar to the argument of the estimate for $P_{\alpha}(t)\varphi$, we can deal with the integral with respect to $\theta$, by using \eqref{eq. integral of halpha} and Lemma~\ref{lem. log constant} (i). Therefore,
	\begin{equation}\label{eq. integral of I}
	\begin{aligned}
	&\int_{0}^{t} (t-s)^{\alpha-1}\int_{0}^{\infty} \theta h_{\alpha}(\theta) \|e^{(t-s)^{\alpha}\theta \Delta} |u(s)|^{p-1}u(s)\|_{p,\tilde{\gamma};T^{\alpha/2}}d\theta ds\\
	&\lesssim \int_{0}^{t} (t-s)^{\alpha-1-\frac{\alpha N}{2}\left( 1-\frac{1}{p} \right)} \left( \LOG{t-s} \right)^{\frac{\tilde{\gamma}}{p}-\tilde{\gamma}}\\
	&\hspace{3cm}\times s^{-\frac{\alpha N}{2}(p-1)}\left( \LOG{s} \right)^{\tilde{\gamma}-\gamma p} ds K^{p}\\
	&=: \int_{0}^{t} I(t,s)ds K^{p}.
	\end{aligned}
	\end{equation}
	Note that $\frac{N(p-1)}{2}=1$ since $p=p_F$. For $0<s<t/2$, by Lemma~\ref{lem. log integral} (ii) with $a=\alpha$ and $q= \gamma p-\tilde{\gamma}\ge 0$, we obtain
	\[
	\begin{aligned}
	&\int_{0}^{t/2}I(t,s)ds\\
	&\lesssim t^{\alpha-1-\frac{\alpha N}{2}\left( 1-\frac{1}{p} \right)} \left( \LOG{t} \right)^{\frac{\tilde{\gamma}}{p}-\tilde{\gamma}}\int_{0}^{t/2}s^{-\alpha} \left( \LOG{s} \right)^{\tilde{\gamma}-\gamma p}ds\\
	&\lesssim (1-\alpha)^{-1} t^{-\frac{\alpha N}{2}\left( 1-\frac{1}{p} \right)} \left(\LOG{t} \right)^{\frac{\tilde{\gamma}}{p}-\gamma p}.
	\end{aligned}
	\]
	For $t/2\le s<t$, again by Lemma~\ref{lem. log integral} (ii) with
	\[
	a= 1-\alpha + \frac{\alpha N}{2} \left( 1-\frac{1}{p} \right)\in (0,1)\ \ \text{and}\ \ q=\tilde{\gamma}-\frac{\tilde{\gamma}}{p},
	\]
	we get
	\[
	\begin{aligned}
	&\int_{t/2}^{t} I(t,s)ds\\
	&\lesssim (1-a)^{-1} t^{-\alpha} \left( \LOG{t} \right)^{\tilde{\gamma}-\gamma p}\int_{t/2}^{t} (t-s)^{\alpha -1 -\frac{\alpha N}{2}\left( 1-\frac{1}{p} \right)} \left( \LOG{t-s} \right)^{\frac{\tilde{\gamma}}{p}-\tilde{\gamma}}ds\\
	&\lesssim \alpha^{-1} t^{-\frac{\alpha N}{2}\left( 1-\frac{1}{p} \right)} \left(\LOG{t} \right)^{\frac{\tilde{\gamma}}{p}-\gamma p}.
	\end{aligned}
	\]
	Note that in this setting,
	\[
	1-a = \alpha \left( 1-\frac{N}{2}\left( 1-\frac{1}{p} \right) \right).
	\]
	Therefore,
	\[
	\begin{aligned}
	&t^{\frac{\alpha N}{2}\left( 1-\frac{1}{p} \right)} \left( \LOG{t} \right)^{\gamma-\frac{\tilde{\gamma}}{p}} \left( \int_{0}^{t/2}+ \int_{t/2}^{t} \right) I(t,s)ds\\
	&\lesssim \alpha^{-1} (1-\alpha)^{-1} \left( \LOG{t} \right)^{\gamma(1-p)}.
	\end{aligned}
	\]
	Thus, $\Phi(\mathcal{B}_{K,T})\subset \mathcal{B}_{K,T}$ is guaranteed provided
	\begin{equation}\label{eq. into ball1}
	C \| \varphi\|_{1,\gamma;T^{\alpha/2}}+ C K^p \alpha^{-1}(1-\alpha)^{-1}\left( \LOG{T} \right)^{\gamma(1-p)}\le K.
	\end{equation}
	This is possible if we take 
	\begin{equation}\label{eq. into ball2}
	K=c_1^{*} \alpha^{\frac{N}{2}}(1-\alpha)^{\frac{N}{2}}\left( \LOG{T} \right)^{\gamma}
	\end{equation}
	and
	\begin{equation}\label{eq. into ball3}
	\|\varphi\|_{1,\gamma;T^{\alpha/2}}=c_2^{*} K \le C \alpha^{\frac{N}{2}}(1-\alpha)^{\frac{N}{2}} \left( \LOG{T} \right)^{\gamma}
	\end{equation}
	for sufficiently small $c_1^*, c_2^{*}>0$.

	\noindent\underline{\textbf{Step 2.}} We prove that $\Phi: \mathcal{B}_{K,T}\to \mathcal{B}_{K,T}$ is a contraction mapping for appropriate $K>0$. For $u,v \in \mathcal{B}_{K,T}$,
	\begin{equation}
	\begin{aligned}
	&\| \Phi(u)(t)-\Phi(v)(t) \|_{p,\tilde{\gamma};T^{\alpha/2}}\\
	&\lesssim \int_{0}^{t} (t-s)^{\alpha-1} \int_{0}^{\infty} \theta h_{\alpha}(\theta) \left\| e^{(t-s)^{\alpha}\theta \Delta }\left( |u(s)|^{p-1}u(s)-|v(s)|^{p-1}v(s) \right)  \right\|_{p,\tilde{\gamma};T^{\alpha/2}} d\theta ds.
	\end{aligned}
	\end{equation}
	By Lemma~\ref{lem. Holder}, we have
	\begin{equation}
	\begin{aligned}
	&\left\| |u(s)|^{p-1}u(s)- |v(s)|^{p-1}v(s) \right\|_{1,\tilde{\gamma}; T^{\alpha/2}}\\
	&\lesssim \left(\left\| u(s)^{p-1} \right\|_{\frac{p}{p-1},\tilde{\gamma} T^{1/2}} + \left\| v(s)^{p-1} \right\|_{\frac{p}{p-1},\tilde{\gamma} T^{1/2}}\right) \|u(s)-v(s)\|_{p,\tilde{\gamma};T^{\alpha/2}}\\
	& = \left( \| u(s)\|^{p-1}_{p,\tilde{\gamma};T^{\alpha/2}} + \| v(s)\|^{p-1}_{p,\tilde{\gamma};T^{\alpha/2}} \right)\|u(s)-v(s)\|_{p,\tilde{\gamma};T^{\alpha/2}}.
	\end{aligned}
	\end{equation}
	This together with Theorem~\ref{theo. semigroup estimate} provides that
	\begin{equation}
	\begin{aligned}
	&\| \Phi(u)(t)-\Phi(v)(t) \|_{p,\tilde{\gamma};T^{\alpha/2}}\\
	&\lesssim \int_{0}^{t} (t-s)^{\alpha-1} \int_{0}^{\infty} \theta h_{\alpha}(\theta) 
	\left[ T^{-\frac{\alpha N}{2}\left(1-\frac{1}{p} \right)} \left( \LOG{T} \right)^{\frac{\tilde{\gamma}}{p}-\tilde{\gamma}}\right.\\
	&\left.\hspace{1cm} + \left( (t-s)^{\alpha}\theta \right)^{-\frac{\alpha N}{2}\left( 1-\frac{1}{p} \right)} \left( \LOG{(t-s)^{\alpha}\theta} \right)^{\frac{\tilde{\gamma}}{p}-\tilde{\gamma}}
	\right]\\
	&\hspace{1cm}\times\left( \| u(s)\|^{p-1}_{p,\tilde{\gamma};T^{\alpha/2}} + \| v(s)\|^{p-1}_{p,\tilde{\gamma};T^{\alpha/2}} \right)\|u(s)-v(s)\|_{p,\tilde{\gamma};T^{\alpha/2}}
	d\theta ds.
	\end{aligned}
	\end{equation}
	Same as Step 1., we obtain
	\begin{equation}
	\begin{aligned}
	&\| \Phi(u)(t)-\Phi(v)(t) \|_{p,\tilde{\gamma};T^{\alpha/2}}\\
	&\lesssim \int_{0}^{t} (t-s)^{\alpha-1-\frac{\alpha N}{2}\left( 1-\frac{1}{p} \right)} \left( \LOG{t-s} \right)^{\frac{\tilde{\gamma}}{p}-\tilde{\gamma}}\\
	&\hspace{1cm} \times s^{-\frac{\alpha N}{2}(p-1)}\left( \LOG{s} \right)^{\tilde{\gamma}-\gamma p} ds K^{p-1} \|u-v\|_{\mathcal{X}_T}\\
	&= \int_{0}^{t} I(t,s)ds K^{p-1}\| u-v\|_{\mathcal{X}_T}.
	\end{aligned}
	\end{equation}
	Therefore, $\Phi : \mathcal{B}_{K,T}\to \mathcal{B}_{K,T}$ is contraction if
	\begin{equation}\label{eq. contraction}
	C K^{p-1}\alpha^{-1}(1-\alpha)^{-1}\left( \LOG{T} \right)^{\gamma(1-p)}\le \frac{1}{2}.
	\end{equation}
	This is possible if we take $c_1^{*}>0$ small enough. By \eqref{eq. into ball1}, \eqref{eq. into ball2}, \eqref{eq. into ball3}, and \eqref{eq. contraction}, we obtain the fixed point $u\in \mathcal{B}_{K,T}$ of $\Phi$.
	
	\noindent\underline{\textbf{Step 3.}} Finally, we prove \eqref{eq. regularity(i)} and \eqref{eq. initial valure (i)}. For \eqref{eq. regularity(i)}, it suffices to show that
	\[
	\| u(t)\|_{1,\gamma ;T^{\alpha/2}} \le C \|\varphi\|_{1,\gamma;T^{\alpha/2}}.
	\]
	Since $u=\Phi(u)$ and $\|P_{\alpha}(t)\|_{1,\gamma;T^{\alpha/2}}\le C \|\varphi\|_{1,\gamma;T^{\alpha/2}}$, it remains to check the nonlinear term. It follows that
	\begin{equation}\label{eq. step3.1}
	\begin{aligned}
	&\| e^{(t-s)^{\alpha}\theta \Delta} |u(s)|^{p-1}u(s) \|_{1,\gamma;T^{\alpha/2}}\\
	&\lesssim \left[ \left(\LOG{T}\right)^{\gamma-\tilde{\gamma}} + \left( \LOG{(t-s)^{\alpha}\theta} \right)^{\gamma-\tilde{\gamma}} \right] \|u(s)\|^{p}_{p,\tilde{\gamma};T^{\alpha/2}}
	\end{aligned}
	\end{equation}
	since $u\in \mathcal{B}_{K,T}$, Theorem~\ref{theo. semigroup estimate} and Lemma~\ref{lem. power and norm}. Moreover, by Lemma~\ref{lem. log constant} (ii), for any $\epsilon >0$, it follows that
	\begin{equation}\label{eq. step3.2}
	\LOG{(t-s)^{\alpha}\theta}\le \left\{
	\begin{aligned}
	& C\theta^{-\epsilon}\LOG{(t-s)^{\alpha}}\ &&\mbox{if}\ \ 0<\theta\le 1,\\
	& \theta \LOG{(t-s)^{\alpha}} &&\mbox{if}\ \ 1<\theta<\infty.
	\end{aligned}
	\right.
	\end{equation}
	We take $\epsilon>0$ sufficiently small so that $\epsilon (\gamma-\tilde{\gamma})<1$. Therefore, by $u\in \mathcal{B}_{K,T}$ and \eqref{eq. integral of halpha},
	\begin{equation}\label{eq. step3.3}
	\begin{aligned}
	&\| u(t)-P_{\alpha}\varphi \|_{1,\gamma;T^{\alpha/2}}\\
	&\lesssim \int_{0}^{t} (t-s)^{\alpha-1} \left( \LOG{(t-s)^{\alpha}} \right)^{\gamma-\tilde{\gamma}} s^{-\frac{\alpha N(p-1)}{2}} \left( \LOG{s} \right)^{\tilde{\gamma}-p\gamma}ds K^{p}\\
	&=:\int_{0}^{t} J(s,t)ds K^{p}.
	\end{aligned}
	\end{equation}
	Firstly, for $0<s<t/2$, by Lemma~\ref{lem. log integral} (ii) with $a=\alpha$ and $q=p\gamma -\tilde{\gamma}$,
	\begin{equation}
	\begin{aligned}
	\int_{0}^{t/2} J(s,t) &\lesssim \int_{0}^{t/2} s^{-\alpha} \left( \LOG{s}\right)^{\tilde{\gamma}-p\gamma } ds t^{\alpha-1} \left( \LOG{t} \right)^{\gamma-\tilde{\gamma}}\\
	&\lesssim (1-\alpha)^{-1} \left( \LOG{T} \right)^{\gamma(1-p)}.
	\end{aligned}
	\end{equation}
	Moreover, for $t/2<s<t$, by Lemma~\ref{lem. log integral} (ii) with $a=1-\alpha$ and $q=\tilde{\gamma}-\gamma$,
	\begin{equation}
	\begin{aligned}
	\int_{t/2}^{t} J(s,t)ds &\lesssim \int_{0}^{t/2} s^{\alpha-1} \left( \LOG{s} \right)^{\gamma-\tilde{\gamma}} ds t^{-\alpha} \left( \LOG{t} \right)^{\tilde{\gamma}-p\gamma}\\
	&\lesssim \alpha^{-1} \left( \LOG{T} \right)^{\gamma(1-p)}.
	\end{aligned}
	\end{equation}
	Hence, by \eqref{eq. into ball1},
	\begin{equation}
	\begin{aligned}
	\| u(t)-P_{\alpha}(t)\varphi\|_{1,\gamma;T^{\alpha/2}}&\le C \alpha^{-1}(1-\alpha)^{-1} \left( \LOG{T} \right)^{\gamma(1-p)} K^{p}\\
	&\le C \|\varphi\|_{1,\gamma;T^{\alpha/2}}.
	\end{aligned}
	\end{equation}
	Furthermore, by repeating the above calculation and using \eqref{eq. into ball1}, \eqref{eq. into ball2}, we get
	\begin{equation}\label{eq. step3 convergence}
	\begin{aligned}
	\| u(t) - P_{\alpha}(t)\varphi\|_{1,\gamma;T^{\alpha/2}} &\lesssim \alpha^{-1}(1-\alpha)^{-1} \left( \LOG{t} \right)^{\gamma(1-p)} K^{p}\\
	&\lesssim \alpha^{N/2} (1-\alpha)^{\frac{N}{2}} \left( \LOG{T} \right)^{\gamma}\left( \LOG{t} \right)^{\gamma(1-p)},
	\end{aligned}
	\end{equation}
	whence follows \eqref{eq. initial valure (i)}.
	
\end{proof}

\subsection{Proof of Theorem~\ref{theorem 2}}

\begin{proof}[Proof of Theorem~\ref{theorem 2}]
	The basic plan of the proof is almost the same as Theorem~\ref{theorem 1}. However, we can use Lemma~\ref{lem. log integral} (i) instead of Lemma~\ref{lem. log integral} (ii), in order to avoid the dependence on $1-\alpha$. We just illustrate estimates for the nonlinear term which includes major differences by using Lemma~\ref{lem. log integral} (i).
	
	\noindent\underline{\textbf{Step 1.}} Let $0\le \tilde{\gamma} <N/2$ be fixed arbitrarily. We apply the fixed point theorem for $\Phi$ in the functional space $\mathcal{X}_{T}$ with $\gamma=N/2$.
	
	Same as the proof of Theorem~\ref{theorem 1}, we shall estimate \eqref{eq. integral of I}. Note that for $\gamma=N/2$,
	\[
	\tilde{\gamma}-\gamma p= \tilde{\gamma} - \gamma\left(1+\frac{2}{N} \right) = \tilde{\gamma}-\gamma -\frac{2}{N}\gamma <-1.
	\]
	Therefore, for $0<s<t/2$, we can use Lemma~\ref{lem. log integral} (i) with $a=\alpha$ and $q=\gamma p-\tilde{\gamma}>1$ to obtain
	\[
	\begin{aligned}
	\int_{0}^{t/2}I(t,s)ds&\lesssim t^{\alpha-1-\frac{\alpha N}{2}\left( 1-\frac{1}{p} \right)} \left( \LOG{t} \right)^{\frac{\tilde{\gamma}}{p}-\tilde{\gamma}}\int_{0}^{t/2}s^{-\alpha} \left( \LOG{s} \right)^{\tilde{\gamma}-\gamma p}ds\\
	&\lesssim  t^{-\frac{\alpha N}{2}\left( 1-\frac{1}{p} \right)} \left(\LOG{t} \right)^{\frac{\tilde{\gamma}}{p}-\gamma p+1} \log\left( e+2t \right)\\
	&\lesssim t^{-\frac{\alpha N}{2}\left( 1-\frac{1}{p} \right)} \left( \LOG{t} \right)^{\frac{\tilde{\gamma}}{p}-\gamma}\log\left( e+2t \right).
	\end{aligned}
	\]
	For $t/2\le s<t$, by Lemma~\ref{lem. log integral} (ii) with
	\[
	a=1-\alpha+ \frac{\alpha N}{2}\left( 1-\frac{1}{p} \right)\ \text{and}\ q=\tilde{\gamma}-\frac{\tilde{\gamma}}{p},
	\]
	it follows that
	\[
	\begin{aligned}
	\int_{t/2}^{t} I(t,s)ds &\lesssim t^{-\alpha} \left( \LOG{t} \right)^{\tilde{\gamma}-\gamma p}\int_{0}^{t/2} s^{\alpha -1 -\frac{\alpha N}{2}\left( 1-\frac{1}{p} \right)} \left( \LOG{s} \right)^{\frac{\tilde{\gamma}}{p}-\tilde{\gamma}}ds\\
	&\lesssim \alpha^{-1} t^{-\frac{\alpha N}{2}\left( 1-\frac{1}{p} \right)} \left(\LOG{t} \right)^{\frac{\tilde{\gamma}}{p}-\gamma p}\\
	&\lesssim \alpha^{-1} t^{-\frac{\alpha N}{2}\left( 1-\frac{1}{p} \right)} \left( \LOG{t} \right)^{\frac{\tilde{\gamma}}{p}-\gamma} \log\left( e+2t \right).
	\end{aligned}
	\]
	Thus, $\Phi(\mathcal{B}_{K,T})\subset \mathcal{B}_{K,T}$ is guaranteed provided
	\begin{equation}\label{eq. into ball1 ii}
	C \| \varphi\|_{1,\gamma;T^{\alpha/2}}+ C K^p \alpha^{-1} \log(e+2T)\le K.
	\end{equation}
	This is possible if we take 
	\begin{equation}\label{eq. into ball2 ii}
	K= c_3^* \alpha^{\frac{N}{2}} \left(\log(e+2T)\right)^{-\frac{N}{2}}
	\end{equation}
	and
	\begin{equation}\label{eq. into ball3 ii}
	\|\varphi\|_{1,\gamma;T^{\alpha/2}} = c_4^* K \le C \alpha^{\frac{N}{2}} \left(\log(e+2T)\right)^{-\frac{N}{2}}
	\end{equation}
	for sufficiently small $c_3^*,\ c_4^*>0$. Moreover, we can see that $\Phi : \mathcal{B}_{K,T}\to \mathcal{B}_{K,T}$ is contraction if
	\begin{equation}\label{eq. contraction ii}
	C K^{p-1} \alpha^{-1} \log \left( e+2T \right)\le \frac{1}{2}.
	\end{equation}
	By \eqref{eq. into ball1 ii}, \eqref{eq. into ball2 ii}, \eqref{eq. into ball3 ii}, and \eqref{eq. contraction ii}, we obtain the fixed-point $u\in \mathcal{B}_{K,T}$ of $\Phi$.
	
	\noindent\underline{\textbf{Step 2.}} It suffices to estimate $\| u(t)-P_{\alpha}(t)\varphi\|_{1,\gamma;T^{\alpha/2}}$. By \eqref{eq. step3.3}, we shall estimate the integral of $J(s,t)$. 
	
	For $0<s<t/2$, it follows that by Lemma~\ref{lem. log integral} (i),
	\begin{equation}
	\begin{aligned}
	\int_{0}^{t/2} J(s,t)dt & \lesssim \int_{0}^{t/2} s^{-\alpha} \left( \LOG{s} \right)^{\tilde{\gamma}-p\gamma}ds  t^{\alpha-1}\left( \LOG{t} \right)^{\gamma-\tilde{\gamma}}\\
	&\lesssim \log\left( e+2T \right).
	\end{aligned}
	\end{equation}
	For $0<s<t/2$, same as Step~3 in the proof of Theorem~\ref{theorem 1},
	\begin{equation}
	\begin{aligned}
	\int_{t/2}^{t} J(s,t)ds &\lesssim \alpha^{-1} \left( \LOG{T} \right)^{\gamma(1-p)}\\
	&\lesssim \alpha^{-1} \log\left( e+2T \right).
	\end{aligned}
	\end{equation}
	Therefore, by \eqref{eq. into ball1 ii},
	\begin{equation}
	\begin{aligned}
	\|u(t)-P_{\alpha}(t)\varphi\|_{1,\gamma;T^{\alpha/2}}&\lesssim \alpha^{-1}\log\left( e+2T \right)K^{p}\\
	&\lesssim \|\varphi\|_{1,\gamma;T^{\alpha/2}}.
	\end{aligned}
	\end{equation}
	Moreover, by \textit{not} using Lemma~\ref{lem. log integral} (i), we have generally \eqref{eq. step3 convergence}. Therefore, by \eqref{eq. step3 convergence}, \eqref{eq. into ball1 ii} and \eqref{eq. into ball2 ii},
	\begin{equation}
	\begin{aligned}
	\| u(t)-P_{\alpha}(t)\varphi \|_{1,\gamma;T^{\alpha/2}}&\lesssim \alpha^{-1}(1-\alpha)^{-1} \left( \LOG{t} \right)^{\gamma(1-p)} K^{p}\\
	&\lesssim \alpha^{\frac{N}{2}} (1-\alpha)^{-1} \left( \log\left( e+2T \right) \right)^{-\frac{N}{2}}\left( \LOG{t} \right)^{\gamma(1-p)},
	\end{aligned}
	\end{equation}
	whence follows \eqref{eq. convergence gamma N/2}.	
	\end{proof}

	\section{Life span estimates}\label{section. Life span}
	In this section, we apply Theorems~\ref{theorem 1} and \ref{theorem 2} to life span estimates for certain initial data. From now on, we only treat nonnegative initial data and solutions. We recall the following supersolution argument introduced by \cite{RobSier2013}. See also \cite{GMS22, LuchYama19}.
	
	\begin{prop}\label{prop. comparison}
		Suppose that there exists a nonnegative supersolution $v$ of \eqref{eq. TFF} on $(0,T)$. Then, there exists a nonnegative solution of \eqref{eq. TFF} $\overline{v}$ on $(0,T)$ such that
		\[
		0\le \overline{v}(t,x)\le v(t,x)
		\]
		for almost all $(t,x)\in (0,T)\times \mathbb{R}^N$. Furthermore, one can take $\overline{v}$ such that $0\le \overline{v}\le u$ for any solution $u$ of \eqref{eq. TFF}.
	\end{prop}
	We call the solution $\overline{v}$ mentioned in Proposition~\ref{prop. comparison} the minimal solution of \eqref{eq. TFF}. The uniqueness of the minimal solution enables us to define the life span of the solution of problem \eqref{eq. TFF}. For each $0<\alpha<1$, set
	\begin{equation}
	T_{\alpha}\left[ \varphi \right] := \sup \left\{ T'>0; \text{there exists a minimal solution of \eqref{eq. TFF} on }(0,T'). \right\}.
	\end{equation}
	We call $T_{\alpha}[\varphi]\ge 0$ life span for initial data $\varphi$. Note that by Proposition~\ref{prop. comparison}, for $T>0$ appearing in Theorems~\ref{theorem 1} and \ref{theorem 2}, it follows that $T\le T_{\alpha}[\varphi]$.
	
	\subsection{Life span estimates for singular initial data}
	
	In this subsection, we shall apply Theorems~\ref{theorem 1} and \ref{theorem 2} to life span estimates for initial data
	\begin{equation}
	f_{\beta}(x):= |x|^{-N} \left( \LOG{|x|} \right)^{-\frac{N}{2}-1-\beta}\quad \mbox{for}\quad -\frac{N}{2}<\beta<\infty.
	\end{equation}
	A similar topic for the heat equation with a nonlinear boundary condition is investigated in \cite{Hisa21}. In order to use Theorems~\ref{theorem 1} and \ref{theorem 2}, we need to calculate the weak Zygmund type norm of $f_{\beta}$.
	
	\begin{prop}\label{prop. fbeta}
		For $-\frac{N}{2}<\beta<\infty$, it follows that
		\[
		f_{\beta}\in \unifz{1}{\frac{N}{2}+\beta}.
		\]
		Moreover, its norm is evaluated as follows:
		\begin{equation}\label{eq. norm N2 beta of fbeta}
		\left\| f_{\beta} \right\|_{1,\frac{N}{2}+\beta;\rho} \le C\log\left( e+ 2\rho \right).
		\end{equation}
	\end{prop}
	
	\begin{proof}
		By \eqref{eq. zygmund norm for cut off}, it holds that
		\[
		\| f_{\beta}\|_{1,\frac{N}{2}+\beta;\rho}
		\le \sup_{0<s<\omega_N\rho^{N}} \left[ \left( \LOG{s} \right)^{\frac{N}{2}+\beta} \int_{0}^{\min\left( s, \omega_N \rho^{N} \right)} f^{*}_{\beta}(\tau)d\tau \right].
		\]
		We easily check that there exists $C=C(N,\beta)$ such that
		\[
		f^{*}_{\beta}(\tau)\le C \tau^{-1} \left( \LOG{\tau} \right)^{-\frac{N}{2}-1-\beta}.
		\]
		Therefore, Lemma~\ref{lem. log integral} (i) with $a=-1$ and $q=1+\beta+N/2$ gives
		\[
		\begin{aligned}
		\| f_{\beta}\|_{1,\frac{N}{2}+\beta;\rho}&\le \sup_{0<s<\omega_N\rho^{N}} \left[ \left( \LOG{s} \right)^{\frac{N}{2}+\beta} \int_{0}^{\min\left( s, \omega_N \rho^{N} \right)} f^{*}_{\beta}(\tau)d\tau \right]\\
		&\lesssim \sup_{0<s<\omega_N\rho^{N}} \left[ \left( \LOG{s} \right)^{\frac{N}{2}+\beta} \log\left( e+ 2s \right) \left( \LOG{s} \right)^{-\frac{N}{2}-\beta} \right]\\
		&\lesssim \log\left( e+2\rho \right).
		\end{aligned}
		\]	
	\end{proof}
	
	In general, for $0\le \gamma_1\le \gamma_2$, it holds that
	\[
	\| f\|_{1,\gamma_1 ;\rho}\le C \left( \LOG{\rho} \right)^{\gamma_1-\gamma_2} \|f\|_{1,\gamma_2;\rho}.
	\]
	Thus, we obtain
	\begin{equation}\label{eq. norm N2 of fbeta}
	\|f_{\beta}\|_{1,\frac{N}{2};\rho}\lesssim \left( \LOG{\rho} \right)^{-\beta} \log\left(e+2\rho\right).
	\end{equation}
	
	For the use of \eqref{eq. Hisa-Kojima necesarry}, the following is convenient:
	\begin{equation}\label{eq. Hisa-Kojima necesarry log ver.}
	\begin{aligned}
	\sup_{z\in\mathbb{R}^N}  \int_{B(z;\sigma)} \varphi(x)\,dx &\le \gamma_1 \left(\int_{\sigma^{2/\alpha}/(16T)}^{1/4} t^{-\alpha} \, dt \right)^{-\frac{N}{2}}\\
	&= \gamma_1 \left(\int_{\sigma^{2/\alpha}/(16T)}^{1/4} t^{1-\alpha} t^{-1} \, dt \right)^{-\frac{N}{2}}\\
	&\le \gamma_1 \left( \frac{T}{\sigma^{2/\alpha}} \right)^{\frac{N}{2}(1-\alpha)} \left( \log\left( e+ \frac{T}{\sigma^{2/\alpha}} \right) \right)^{-\frac{N}{2}},
	\end{aligned}
	\end{equation}
	for all $0<\sigma < T^{\alpha/2}$.
	
	Observations in this subsection are divided into the following three cases:
	\begin{itemize}
		\item $\beta>0$ (Solvable even when $\alpha \nearrow 1$),
		
		\item $\beta=0$ (Threshold),
		
		\item $-\frac{N}{2}<\beta<0$ (Unsolvable when $\alpha\nearrow 1$).
	\end{itemize}
	
	In what follows, we are in particular interested in the life span estimate as $\alpha \nearrow 1$, so we shall use $C=C(\alpha)$ to denote generic positive constants which may depend on $\alpha\in (0,1)$ and
	\[
	\limsup_{\alpha \nearrow1 } C(\alpha) \in (0,\infty).
	\]
	
	Firstly, we shall observe the behavior of $T_{\alpha} [\kappa f_{\beta}]$ for $\beta>0$ and parameter $\kappa>0$. In both cases where $\kappa>0$ is large and small, optimal estimates with respect to $\kappa>0$ is deduced. Furthermore, as $\alpha \nearrow 1$, these estimates for $T_{\alpha}[\kappa f_{\beta}]$ with $0<\alpha<1$ correspond to the optimal life span estimates for the classical problem \eqref{eq. Fujita} with $\varphi = \kappa f_{\beta}$. See Remarks~\ref{rem. beta positive: kappa large} and \ref{rem. beta positive: kappa small} for details.

	\begin{theo}\label{theo. beta positive: kappa large}
		Suppose that $\beta>0$. Let $\kappa >0$ be sufficiently large. Then, there exist constants $C_1, C_2, D_1, D_2>0$ such that,
		\begin{equation}
		\begin{aligned}
		&\max\left( \exp\left( -C_1(\alpha) \kappa^{{\frac{2}{N+2\beta}}} \right),\ \exp\left( -D_1 \kappa^{\frac{1}{\beta}} \right) \right)\\
		& \le T_{\alpha}[\kappa f_{\beta}]\\
		&\le \min\left( \exp\left( - C_2 \kappa^{\frac{2}{N+2\beta}} \right),\ E(\kappa ,\alpha)\exp\left( -D_2 \kappa^{\frac{1}{\beta}} \right) \right).
		\end{aligned}
		\end{equation}
		Here,
		\[
		\begin{aligned}
		C_1(\alpha)&=C (1-\alpha)^{-\frac{N}{N+2\beta}},\\
		E(\kappa, \alpha)&= C\left( 4^{\alpha -1}- C(1-\alpha)
		\kappa^{\frac{1}{\beta}} \right)_{+}^{-\frac{1}{1-\alpha}}.
		\end{aligned}
		\]
		In particular, when $\alpha \nearrow 1$,
		\begin{equation}
		\exp\left( -D_1 \kappa^{\frac{1}{\beta}} \right)\le \liminf_{\alpha \nearrow 1}T_{\alpha}[\kappa f_{\beta}]\le \limsup_{\alpha\nearrow 1} T_{\alpha}[\kappa f_{\beta}]\le \exp\left( -D_2 \kappa^{\frac{1}{\beta}} \right)
		\end{equation}
		for large $\kappa>0$.
	\end{theo}
	
	\begin{proof}
		Firstly we deduce the lower estimates. By Theorem~\ref{theorem 1} and \eqref{eq. norm N2 beta of fbeta}, if
		\begin{equation}\label{eq. Prf of Thm3 1}
		\log\left( e+2T \right) \left( \LOG{T} \right)^{-\frac{N}{2}-\beta}\le C\kappa^{-1}(1-\alpha)^{\frac{N}{2}},
		\end{equation}
		then the solvability on $(0,T)$ with $u(0)=\kappa f_{\beta}$ is guaranteed. For appropriate $C_1'>0$ and large $\kappa>0$, it is verified that
		\[
		T= \exp\left( -C_1' (1-\alpha)^{-\frac{N}{N+2\beta}} \kappa^{\kappa^{\frac{2}{N+2\beta}}} \right)
		\]
		satisfies \eqref{eq. Prf of Thm3 1}. Moreover, by Theorem~\ref{theorem 2} and \eqref{eq. norm N2 of fbeta}, 
		\begin{equation}\label{eq. Prf of Thm3 2}
		\left( \log\left( e+2T \right) \right)^{1+\frac{N}{2}} \left( \LOG{T} \right)^{-\beta}\le C\kappa^{-1}
		\end{equation}
		gives the solvability on $(0,T)$. We can check that
		\[
		T= \exp\left( -D_1 \kappa^{\frac{1}{\beta}} \right)
		\]
		fulfills \eqref{eq. Prf of Thm3 2} for appropriate $D_1>0$ and sufficiently large $\kappa>0$. Therefore, the desired lower estimates are proved.
		
		Next, we show the upper estimates. Using \eqref{eq. Hisa-Kojima necesarry} with $T=T_{\alpha}[\kappa f_{\beta}]/2$, $\sigma^{2/\alpha}=T/2$ and
		\begin{equation}\label{eq. lower estimate of int of fbeta}
		\begin{aligned}
		\int_{B(0;\rho)} f_{\beta}(x)\,dx &=C \int_{\LOG{\rho}}^{\infty} \frac{e^{\tau}}{e^{\tau}-e}\tau^{-\frac{N}{2}-1-\beta}d\tau\\
		&\ge C \left( \LOG{\rho} \right)^{-\frac{N}{2}-\beta},
		\end{aligned}
		\end{equation}
		we deduce that
		\begin{equation}
		\left( \LOG{T} \right)^{-\frac{N}{2}-\beta}\le C \kappa^{-1}.
		\end{equation}
		This estimate yields that
		\[
		T_{\alpha}[\kappa f_{\beta}]\le \exp\left( -C_2 \kappa^{\frac{2}{N+2\beta}} \right)
		\]
		for appropriate $C_2>0$. Furthermore, \eqref{eq. Hisa-Kojima necesarry} with
		\[
		T=\frac{T_{\alpha}[\kappa f_{\beta}]}{2},\ \ \sigma^{2/\alpha}=\frac{1}{2} \exp\left( -D \kappa^{\frac{1}{\beta}} \right)
		\]
		for $D>D_1$ and \eqref{eq. lower estimate of int of fbeta} provides that
		\[
		\begin{aligned}
		&\frac{1}{1-\alpha}\left[ 4^{\alpha-1}- 16^{\alpha-1} T_{\alpha}[\kappa f_{\beta}]^{\alpha-1} \exp\left( -\frac{\alpha D}{2}(1-\alpha)\kappa^{\frac{1}{\beta}} \right)  \right]\\
		&= \int_{\frac{\sigma^{2/\alpha}}{16T}}^{\frac{1}{4}} t^{-\alpha}dt\le C \kappa^{-\frac{2}{N}} \left( \LOG{\sigma} \right)^{1+\frac{2\beta}{N}}\le C \kappa^{\frac{1}{\beta}}.
		\end{aligned}
		\]
		Therefore,
		\[
		T_{\alpha}[\kappa f_{\beta}]\le C \left[ 4^{\alpha -1} -C(1-\alpha) \kappa^{\frac{1}{\beta}} \right]^{-\frac{1}{1-\alpha}} e^{-\frac{\alpha D}{2}\kappa^{\frac{1}{\beta}}}
		\]
		provided $4^{\alpha-1}-C(1-\alpha)\kappa^{\frac{1}{\beta}}>0$.
	\end{proof}
	
	\begin{rem}\label{rem. beta positive: kappa large}
		{\rm 
			Theorem~\ref{theo. beta positive: kappa large} states that the optimal decay rate of $T_{\alpha}[\kappa f_{\beta}]$ as $\kappa \to \infty$ is $\exp\left( - C \kappa^{\frac{2}{N+2\beta}} \right)$ provided $0<\alpha<1$, although formally when $\alpha= 1$, the optimal rate turns into $\exp\left(-C \kappa^{\frac{1}{\beta}}\right)$. Note that $ \frac{2}{N+2\beta}<\frac{1}{\beta} $. Such a \textit{discontinuous} change of the optimal rate represents the significant difference between the time-fractional problem \eqref{eq. TFF} and the classical one \eqref{eq. Fujita}.
		}
	\end{rem}
	
	\begin{theo}\label{theo. beta positive: kappa small}
		Suppose that $\beta>0$. Let $\kappa>0$ be sufficiently small. Then, there exist $C_1, C_2, D_1, D_2>0$ such that,
		\begin{equation}
		\begin{aligned}
		&\max\left( \exp\left( C_1(\alpha) \kappa^{-1} \right),\ \exp\left( D_1\kappa^{-\frac{2}{N+2}} \right) \right)\\
		&\le T_{\alpha}[\kappa f_{\beta}]\\
		&\le \min \left( \exp\left( C_2 \kappa^{-1} \right),\ \ E(\kappa, \alpha)\exp\left( D_2 \kappa^{-\frac{2}{N+2}} \right) \right).
		\end{aligned}
		\end{equation}
		Here,
		\[
		\begin{aligned}
		C_1(\alpha)&=C(1-\alpha)^{\frac{N}{2}},\\
		E(\kappa, \alpha)&=C\left(4^{\alpha-1}-C(1-\alpha)\kappa^{-\frac{2}{N+2}}\right)_+^{-\frac{1}{1-\alpha}}.
		\end{aligned}
		\]
		In particular, when $\alpha \nearrow 1$,
		\[
		\exp\left( D_1 \kappa^{-\frac{2}{N+2}} \right) \le \liminf_{\alpha \nearrow 1} T_{\alpha}[\kappa f_{\beta}]\le \limsup_{\alpha\nearrow1} T_{\alpha}[\kappa f_{\beta}]\le \exp\left( D_2 \kappa^{-\frac{2}{N+2}} \right)
		\]
		for small $\kappa>0$.
	\end{theo}
	
	\begin{proof}
		We first check the lower estimates. By Theorem~\ref{theorem 1} and \eqref{eq. norm N2 beta of fbeta}, it suffices to prove that $T=\exp\left( C_1'(1-\alpha)^{\frac{N}{2}} \kappa^{-1} \right)$ satisfies
		\begin{equation}
		\log\left( e+2T \right)\le C\kappa^{-1}(1-\alpha)^{\frac{N}{2}} \left( \LOG{T} \right)^{\frac{N}{2}+\beta}
		\end{equation}
		for appropriate $C_1'>0$ and sufficiently small $\kappa>0$. This estimate can be verified by simple calculation. Moreover, by using Theorem~\ref{theorem 2} and \eqref{eq. norm N2 of fbeta}, it suffices to check that $T=\exp\left( D_1 \kappa^{-\frac{2}{N+2}} \right)$ fulfills
		\begin{equation}
		\left( \log\left( e+2T \right) \right)^{1+\frac{N}{2}}\le C\kappa^{-1} \left( \LOG{T} \right)^{\beta}
		\end{equation}
		for some $D_1>0$ and small $\kappa>0$. This is also proved easily.
		
		Next, we check the upper estimates. By \eqref{eq. Hisa-Kojima necesarry} with $T=T_{\alpha}[\kappa f_{\beta}]/2$ and $\sigma^{2/\alpha}=T/2$, it holds that
		\begin{equation}
		\int_{B\left( 0; \left(T/2\right)^{\alpha/2} \right)} f_{\beta}(x)\,dx\le C \kappa^{-1}.
		\end{equation}
		Note that for $\rho>1$
		\[
		\begin{aligned}
		\int_{B(0;\rho)} f_{\beta}(x)\,dx &\ge C \int_{1/2}^{\rho} r^{-1} \left( \LOG{r} \right)^{-\frac{N}{2}-\beta -1}dr\\
		&\ge \left( \LOG{2^{-1}} \right)^{-\frac{N}{2}-\beta-1} \log \rho\\
		&=C\log \rho.
		\end{aligned}
		\]
		Therefore, for $\kappa>0$ small enough so that $T_{\alpha}[\kappa f_{\beta}]>1$, it follows that $T_{\alpha}[\kappa f_{\beta}]\le \exp\left( C_2\kappa^{-1} \right)$. Moreover, \eqref{eq. Hisa-Kojima necesarry} with
		\[
		T=\frac{T_{\alpha}[\kappa f_{\beta}]}{2},\ \ \sigma^{2/\alpha}=\frac{1}{2}\exp\left( D\kappa^{-\frac{2}{N+2}} \right)
		\]
		for $D<D_1$ gives
		\begin{equation}
		\begin{aligned}
		&\frac{1}{1-\alpha}\left[ 4^{\alpha-1} -16^{\alpha-1} T_{\alpha}[\kappa f_{\beta}]^{\alpha-1} \exp\left( \frac{\alpha D}{2}(1-\alpha)\kappa^{-\frac{2}{N+2}} \right) \right]\\
		&\le C \kappa^{-\frac{2}{N}} \left(\log \sigma \right)^{-\frac{2}{N}}\le C \kappa^{-\frac{2}{N+2}}.
		\end{aligned}
		\end{equation}
		This inequality deduces the desired estimate.
	\end{proof}

	\begin{rem}\label{rem. beta positive: kappa small}
		{\rm 
			Same as Theorem~\ref{theo. beta positive: kappa small}, the \textit{discontinuous} change of the optimal rate with respect to $\kappa\to 0$ is observed in Theorem~\ref{theo. beta positive: kappa large}. Indeed, the optimal decay rate of $T_{\alpha}[\kappa f_{\beta}]$ as $\kappa \to 0$ is $\exp\left( - C \kappa^{-1} \right)$ for $0<\alpha<1$, while when $\alpha= 1$, the optimal rate is $\exp\left(-C \kappa^{\frac{2}{N+2}}\right)$.
		}
	\end{rem}
	
	Next, we investigate the life span estimate for $f_{\beta}$ with $\beta=0$, which represents the threshold singularity for the solvability of \eqref{eq. Fujita}. Indeed, there exists $\tilde{\kappa}>0$ such that:
	\begin{itemize}
		\item If $\kappa<\tilde{\kappa}$, then \eqref{eq. Fujita} with $\varphi=\kappa f_0$ possesses a nonnegative local-in-time solution;
		\item If $\kappa >\tilde{\kappa}$, then \eqref{eq. Fujita} with $\varphi=\kappa f_{0}$ possesses no nonnegative local-in-time solutions.
	\end{itemize}
	See \cite[Corollary~1.2]{HisaIshige18}. Although for the time-fractional problem \eqref{eq. TFF}, it holds that $T_{\alpha}[\kappa f_0]>0$ for any $\kappa>0$, the singularity $f_0$ can be also interpreted as the threshold whether $T_{\alpha}[\kappa f_0]$ decays to zero as $\alpha\nearrow 1$.
	
	\begin{theo}\label{theo. beta zero}
		Suppose that $\beta=0$. Then, there exists $\kappa^*>0$ such that,
		\begin{itemize}
			\item If $\kappa<\kappa^*$, then
			\begin{equation}\label{eq. threshold liminf}
			\liminf_{\alpha \nearrow 1} T_{\alpha}\left[ \kappa f_{0} \right]>0;
			\end{equation}
			
			\item If $\kappa>\kappa^{*}$, then 
			\begin{equation}\label{eq. threshold lim}
			\lim_{\alpha \nearrow 1} T_{\alpha}\left[ \kappa f_{0} \right] =0.
			\end{equation}
		\end{itemize}
		Furthermore, for sufficiently large $\kappa>\kappa^{*}$, it follows that
		\begin{equation}\label{eq. estimate for alpha threshold}
		\log T_{\alpha}\left[ \kappa f_0 \right]\simeq -\kappa^{\frac{2}{N}}(1-\alpha)^{-1}
		\end{equation}
		near $\alpha=1$.
	\end{theo}
	
	\begin{proof}
		Firstly, we observe the existence of the threshold $\kappa=\kappa^{*}$. Let $\kappa >0$ be sufficiently small so that there exists $T>0$ satisfying
		\begin{equation}
		\log\left( e+2T \right)^{1+\frac{N}{2}} = C\kappa^{-1},
		\end{equation}
		where $C>0$ is the constant used in \eqref{eq. main notvanish}. Therefore, By Theoremo~\ref{theorem 2} and \eqref{eq. norm N2 of fbeta}, \eqref{eq. threshold liminf} is guaranteed for small $\kappa>0$. On the other hand, by \eqref{eq. Hisa-Kojima necesarry log ver.} with $T:= \min\left( 1, T_{\alpha}[\kappa f_0] \right)$ and $\sigma := T^{\alpha}$ so that $ T/(\sigma^{2/\alpha})=T$, we get,
		\[
		T^{\frac{N}{2}(1-\alpha)} \le C \kappa^{-1}.
		\]
		This implies that
		\[
		T= \min \left(1, T_{\alpha}[\kappa f_0] \right)\le \left( C\kappa^{-1} \right)^{\frac{N}{2(1-\alpha)}} \to 0
		\]
		as $\alpha \nearrow 1$, if $\kappa>0$ is large so that $C\kappa^{-1}<1$. Therefore, by Proposition~\ref{prop. comparison}, the threshold is obtained by
		\[
		\begin{aligned}
		\kappa^{*} &:= \sup \left\{ \kappa>0; \liminf_{\alpha \nearrow 1}T_{\alpha}[\kappa f_0]>0 \right\}\\
		&:= \inf\left\{ \kappa >0; \lim_{\alpha \nearrow 1} T_{\alpha}[\kappa f_0]=0 \right\}.
		\end{aligned}
		\]
		Indeed, assume that there exist $\kappa_1<\kappa_2$ such that
		\[
		\lim_{\alpha \nearrow 1} T_{\alpha}[\kappa_1 f_0]=0,\quad \liminf_{\alpha \nearrow 1}T_{\alpha}[\kappa_2 f_0]>0.
		\]
		In particular, there exists $\alpha^* \in (0,1)$ such that
		\[
		T_{\alpha^*}[\kappa_1 f_0]< T_{\alpha^*}[\kappa_2f_0].
		\]
		On the other hand, Proposition~\ref{prop. comparison} yields that
		\[
		T_{\alpha^*}[\kappa_1 f_0]\ge  T_{\alpha^*}[\kappa_2f_0]
		\]
		which is the desired contradiction.
		
		Next, we show the estimate \eqref{eq. estimate for alpha threshold}. First, we deduce the lower estimate by using Theorem~\ref{theorem 2}. Indeed, by \eqref{eq. norm N2 beta of fbeta},
		\begin{equation}\label{eq. proof of threshold1}
		\kappa\log\left(e+2T\right) \le C (1-\alpha)^{\frac{N}{2}} \left( \LOG{T} \right)^{\frac{N}{2}}
		\end{equation}
		for appropriate $C>0$ gives the solvability until $T>0$. Substituting
		\[
		T= \exp\left( -C \kappa^{\frac{2}{N}} \left( 1-\alpha \right)^{-1} \right)
		\]
		into \eqref{eq. proof of threshold1}, we verify the lower estimate. Indeed, we easily observe that
		\[
		\left( \LOG{T} \right)^{-\frac{N}{2}}\le \left( \log\left( \frac{1}{T} \right) \right)^{-\frac{N}{2}}= C \kappa^{-1} (1-\alpha)^{\frac{N}{2}}
		\]
		and
		\[
		\log\left( e+2T \right)\le \log\left( e+2 \right).
		\]
		Thus, it suffices to take an appropriate $C>0$.
		
		Finally, we prove the upper estimate by using \eqref{eq. Hisa-Kojima necesarry log ver.} with
		\[
		\sigma^{2/\alpha} =e^{-\frac{1}{1-\alpha}} T_{\alpha}[\kappa f_0].
		\]
		Then, \eqref{eq. Hisa-Kojima necesarry log ver.} means that
		\begin{equation}
		\kappa \left( \LOG{\sigma} \right)^{-\frac{N}{2}} \lesssim \left( \log\left( e+ e^{\frac{1}{1-\alpha}} \right) \right)^{-\frac{N}{2}}\lesssim (1-\alpha)^{\frac{N}{2}}.
		\end{equation}
		Therefore, if $1-\alpha$ is sufficiently small and $\kappa$ is large,
		\begin{equation}
		\begin{aligned}
		T_{\alpha}[\kappa f_0]&\le e^{(1-\alpha)^{-1}} \frac{ \exp\left(-C\kappa^{\frac{2}{N}}(1-\alpha)^{-1}\right) }{ 1- \exp\left( 1- C\kappa^{\frac{2}{N}(1-\alpha)^{-1}} \right)  }\\
		&\lesssim \exp\left( -C \kappa^{\frac{2}{N}}(1-\alpha)^{-1} + (1-\alpha)^{-1}  \right)\\
		&\lesssim \exp\left( -C' \kappa^{\frac{2}{N}}(1-\alpha)^{-1}\right).
		\end{aligned}
		\end{equation}
	\end{proof}

	\begin{rem}
		{\rm
		It is a very interesting, but difficult (maybe hopeless to solve) problem whether $\tilde{\kappa}=\kappa^*$ or not. Indeed, finding specific values appearing in \eqref{eq. main vanish}, \eqref{eq. main notvanish}, and \eqref{eq. Hisa-Kojima necesarry} is extremely hard work. Even if we find the concrete quantities, since we define $\tilde{\kappa}$ and $\kappa^*$ as supremums, or infimums of certain sets, we can only derive the estimates of thresholds.
		}
	\end{rem}

	Finally, we are concerned with the case $\beta<0$, which corresponds to the unsolvability of \eqref{eq. Fujita}. Although Hisa and the author \cite{HisaKojima24} investigated the life span estimate for the cut-off initial data (see Remark~\ref{rem. beta negative} below), we shall deduce the analogous estimate by using Theorem~\ref{theorem 1} for the completeness of the observation in this section.
	
	\begin{theo}\label{theo. beta negative}
		Suppose that $-\frac{N}{2}<\beta<0$. Then, it follows that
		\begin{equation}\label{eq. life span beta negative}
		\log T_{\alpha}\left[ \kappa f_\beta \right]\simeq - \kappa^{\frac{2}{\kappa+2\beta}} (1-\alpha)^{-\frac{N}{N+2\beta}}
		\end{equation}
		near $\alpha=1$ and sufficiently large $\kappa>0$.
	\end{theo}
	\begin{proof}
		We shall derive the lower estimate by using Theorem~\ref{theorem 1}. By \eqref{eq. norm N2 beta of fbeta} it suffices to show that for appropriate $D>0$,
		\[
		T = \frac{1}{\exp\left( D\kappa^{\frac{2}{N+2\beta}} (1-\alpha)^{-\frac{N}{N+2\beta}} \right) -e } 
		\]
		satisfies
		\begin{equation}
		\log\left( e+2T \right) \le  C\kappa^{-1} (1-\alpha)^{\frac{N}{2}} \left( \LOG{T} \right)^{\frac{N}{2}+\beta}
		\end{equation}
		where $C$ is constant appeared in Theorem~\ref{theorem 1}. Indeed, near $\alpha=1$,
		\[
		\begin{aligned}
		\log\left( e+2T \right) \left( \LOG{T} \right)^{-\frac{N}{2}-\beta}&= D^{-\frac{N}{2}-\beta} \kappa^{-1} (1-\alpha)^{\frac{N}{2}} \log\left( e+2T \right)\\
		&\le D^{-\frac{N}{2}-\beta}\log\left( e+2 \right) \kappa^{-1} (1-\alpha)^{\frac{N}{2}}.
		\end{aligned}
		\]
		
		Conversely, \eqref{eq. Hisa-Kojima necesarry log ver.} with $\sigma^{2/\alpha}= e^{-\frac{1}{1-\alpha}} T_{\alpha}[\kappa f_\beta]$ provides
		\begin{equation}
		\left( \LOG{\sigma} \right)^{-\frac{N}{2}-\beta}\le C \kappa^{-1} (1-\alpha)^{\frac{N}{2}}.
		\end{equation}
		Therefore,
		\[
		\begin{aligned}
		&T_{\alpha}[\kappa f_{\beta}]^{\frac{\alpha}{2}}\\
		&\le \frac{e^{\frac{1}{1-\alpha}}}{\exp\left( C\kappa^{\frac{2}{N+2\beta}} (1-\alpha)^{-\frac{N}{N+2\beta}} \right) -e }\\
		&\le \frac{1}{1- e^{1- C\kappa^{\frac{2}{N+2\beta}}} (1-\alpha)^{-\frac{N}{N+2\beta}} }\cdot \exp\left( -C \kappa^{\frac{2}{N+2\beta}} (1-\alpha)^{-\frac{N}{N+2\beta}} + \frac{1}{1-\alpha}\right).
		\end{aligned}
		\]
		On the other hand,
		\[
		\begin{aligned}
		&-C \kappa^{\frac{2}{N+2\beta}} (1-\alpha)^{-\frac{N}{N+2\beta}} +(1-\alpha)^{-1}\\
		&=\kappa^{\frac{2}{N+2\beta}}(1-\alpha)^{-\frac{N}{N+2\beta}} \left(-C+ \kappa^{-\frac{2}{N+2\beta}} (1-\alpha)^{-\frac{2\beta}{N+2\beta}} \right).
		\end{aligned}
		\]
		Hence, we verify the desired estimate near $\alpha=1$ and sufficiently large $\kappa>0$.
		
	\end{proof}
	
	\begin{rem}\label{rem. beta negative}
		{\rm
			Hisa and the author \cite{HisaKojima24} investigated the life span estimate for initial data
			\[
			g_{\beta}(x):= |x|^{-N} \left| \log|x| \right|^{-\frac{N}{2}-1-\beta} \chi_{B(0;e^{-3/2})}
			\]
			with $-N/2<\beta<0$. Note that $g_{\beta}\in L^{1}(\mathbb{R}^N)$ but $f_{\beta}\notin L^{1}(\mathbb{R}^N)$. Therefore, it follows that $T_{\alpha}[\kappa g_{\epsilon}]=\infty$ for small enough $\kappa>0$, while $T_{\alpha}[\kappa f_{\beta}]<\infty$ for any $\kappa>0$. On the other hand, for each $\kappa>0$, near $\alpha=1$, it follows that
			\[
			\log T_{\alpha}[\kappa g_{\beta}]\simeq - \kappa^{\frac{2}{N+2\beta} } (1-\alpha)^{-\frac{N}{N+2\beta}},
			\] 
			which is the same order as \eqref{eq. life span beta negative}.
		}
	\end{rem}

	\subsection{Lee-Ni type estimate}
	
	In this subsection, we deduce a life span estimate for the initial data
	\[
	\phi_A(x):=(1+|x|)^{-A}.
	\]
	Lee and Ni \cite{LeeNi92} obtained detailed life span estimates for \eqref{eq. Fujita} with $u(0)=\phi_A$, and Hisa and Ishige \cite{HisaIshige18} introduce the alternative proof via the necessary and sufficient conditions for \eqref{eq. Fujita}. We apply Theorems~\ref{theorem 1}, \ref{theorem 2} and \eqref{eq. Hisa-Kojima necesarry} to obtain the optimal life span estimate. In particular, we mainly focus on the case $A=N$.
	
	\begin{theo}\label{theo. LeeNi}
		For sufficiently small $\kappa>0$, it holds that
		\begin{equation}
		\begin{aligned}
		&\max\left( \exp\left( C_1(\alpha) \kappa^{-1} \right), \exp\left( D_1\kappa^{-\frac{2}{N+2}} \right) \right)\\
		&\le T_{\alpha}[\kappa \phi_N]\\
		&\le 
		\min \left( \exp\left( C_2 \kappa^{-1} \right), E(\kappa ,\alpha)\exp\left( D_2\kappa^{-\frac{2}{N+2}} \right) \right).
		\end{aligned}
		\end{equation}
		Here,
		\[
		\begin{aligned}
		C_1(\alpha)&=C(1-\alpha)^{\frac{N}{2}},\\
		E(\kappa, \alpha)&=C\left(4^{\alpha-1}-C(1-\alpha)\kappa^{-\frac{2}{N+2}}\right)_+^{-\frac{1}{1-\alpha}}.
		\end{aligned}
		\]
		In particular, when $\alpha \nearrow 1$,
		\[
		\exp\left( D_1 \kappa^{-\frac{2}{N+2}} \right) \le \liminf_{\alpha \nearrow 1} T_{\alpha}[\kappa \phi_N]\le \limsup_{\alpha \nearrow1} T_{\alpha}[\kappa \phi_N]\le \exp\left( D_2 \kappa^{-\frac{2}{N+2}} \right)
		\]
		for small $\kappa>0$.
	\end{theo}
	
	\begin{proof}
		For simplicity, we omit $\phi=\phi_N$. We shall calculate the uniformly local weak Zygmund norm of $\phi$. Note that
		\begin{equation}
		\| \phi\|_{1,\frac{N}{2};\rho}\le \sup_{0<s<\omega_N\rho^{N}} \left[ \left( \LOG{s} \right)^{\frac{N}{2}} \int_{0}^{s} \phi^{*}(\tau)d\tau \right]
		\end{equation}
		by \eqref{eq. zygmund norm for cut off} and
		\begin{equation}
		\phi^{*}(\tau) = \left( 1+ \left( \frac{\tau}{\omega_N} \right)^{\frac{1}{N}} \right)^{-N}.
		\end{equation}
		It holds that for $0<s<1/2$,
		\[
		\int_{0}^{s} \phi^{*}(\tau)d\tau\le s
		\]
		and for $s>1/2$,
		\[
		\begin{aligned}
		\int_{0}^{s} \phi^{*}(\tau) d\tau &\lesssim \int_{0}^{s} \left( 1+ \frac{\tau}{\omega_N} \right)^{-1}d\tau\\
		&\lesssim \log\left( 1+ \frac{s}{\omega_N} \right).
		\end{aligned}
		\]
		Therefore, for sufficiently large $\rho>0$,
		\begin{equation}
		\| \phi\|_{1,\frac{N}{2};\rho}\lesssim \max\left( A, B \right)
		\end{equation}
		where
		\begin{equation}
		\begin{aligned}
		A&:= \sup_{0<s<\frac{1}{2}} \left[ \left( \LOG{s} \right)^{\frac{N}{2}} s\right]\\
		&=\text{Constant (independent of }\rho\text{)}
		\end{aligned}
		\end{equation}
		and
		\begin{equation}
		\begin{aligned}
		B&:= \sup_{\frac{1}{2}<s<\omega_N\rho^{N}} \left[ \left(\LOG{s}\right)^{\frac{N}{2}} \log\left( 1+ \frac{s}{\omega_N} \right) \right]\\
		&\le \left( \LOG{2^{-1}} \right)^{\frac{N}{2}} \log\left( 1+ \rho^{N} \right).
		\end{aligned}
		\end{equation}
		Hence, we can take sufficiently large $T>0$ so that
		\begin{equation}
		\| \phi\|_{1,\frac{N}{2};T^{\frac{\alpha}{2}}}\lesssim \log T .
		\end{equation}
		By the same argument, we can take large $T>0$ so that
		\begin{equation}
		\| \phi\|_{1, \frac{N}{2}+\beta ; T^{\frac{\alpha}{2}}}\lesssim \log T.
		\end{equation}
		Thus, by Theorem~\ref{theorem 1}, the condition
		\begin{equation}\label{eq. proof of LeeNi}
		\kappa \log T \le C (1-\alpha)^{\frac{N}{2}} \left( \LOG{T} \right)^{\frac{N}{2}+\beta}
		\end{equation}
		gives the solvability of \eqref{eq. TFF} with $u(0)=\kappa \phi$ on $(0,T)$ for large enough $T>0$. It is easily verified that $T= \exp\left( C_1(\alpha) \kappa^{-1} \right)$ satisfies \eqref{eq. proof of LeeNi} for small $\kappa>0$. 
		
		Moreover, by Theorem~\ref{theorem 2}, the condition
		\begin{equation}\label{eq. proof of LeeNi 2}
		\kappa \log T \le C \left( \log\left( e+2T\right) \right)^{-\frac{N}{2}}
		\end{equation}
		gives the solvability of \eqref{eq. TFF} with $u(0)=\kappa \phi$ on $(0,T)$ for large $T>0$. We easily check that $T=\exp \left( D_1\kappa^{-\frac{2}{N+2}}\right)$ fulfills \eqref{eq. proof of LeeNi 2} with appropriate $D_1>0$ and sufficiently small $\kappa>0$.
		
		Conversely, the upper estimate is guaranteed by a similar argument to the proof of Theomrem~\ref{theo. beta positive: kappa small}, by using \eqref{eq. Hisa-Kojima necesarry} and
		\[
		\int_{B(0, \rho)} \phi(x)\,dx\ge C \log \left( 1 + \rho \right)
		\]
		for $\rho>1$.
		
	\end{proof}
	
	\begin{rem}
		{\rm
			Formally, let $T_1[\varphi]$ be a life span for \eqref{eq. Fujita} with initial data $u(0)=\varphi$. For \eqref{eq. Fujita} with $u(0)=\kappa \phi$ and $p=p_F$, it holds that
			\begin{equation}\label{eq. LeeNi for Fujita}
			\log T_1[\kappa \phi_N] \simeq \kappa^{-\frac{2}{N+2}}
			\end{equation}
			for small enough $\kappa>0$. See \cite[Theorem~3.15, 3.21]{LeeNi92} and \cite[Theorem~5.1, 5.2]{HisaIshige18}. Theorem~\ref{theo. LeeNi} reflects the difference of the optimal rate of $T_{\alpha}[\kappa \phi]$ with respect to $\kappa>0$, which is also observed in Theorem~\ref{theo. beta positive: kappa large} and Theorem~\ref{theo. beta positive: kappa small} for $u(0)=\kappa f_{\beta}$. Althoug the optimal rate for $0<\alpha <1$ is $\log T_{\alpha}[\kappa\phi] \simeq \kappa^{-1}$, it turns into $\log T_{1}[\kappa \phi]\simeq \kappa^{-\frac{2}{N+2}}$ when $\alpha =1$, which coincides with \eqref{eq. LeeNi for Fujita}.
		}
	\end{rem}

	\begin{rem}
		{\rm 
		Consider the problem \eqref{eq. TFF} with $p=p_F$ and $u(0)=\phi_A$ for $A>N$. In this case, it follows that there exists a threshold $\kappa^*=\kappa^{*}(\alpha)>0$ such that,
		\begin{itemize}
			\item if $0<\kappa < \kappa^{*}$, then $T_{\alpha}[\kappa \phi_A]=\infty$ ;
			\item if $\kappa^* < \kappa$, then $T_{\alpha}[\kappa \phi_A]<\infty$.
		\end{itemize}
		Moreover, we see that $\kappa^*(\alpha)\simeq (1-\alpha)^{N/2}$. See \cite[Theorem~1.5]{HisaKojima24}.
		
		On the other hand, for \eqref{eq. Fujita} with $p=p_F$ and $u(0)=\phi_A$, it is known that
		\[
		\log T_1[\kappa \phi_A] \simeq \kappa^{-\frac{2}{N}}
		\]
		for sufficiently small $\kappa>0$ (see \cite{LeeNi92, HisaIshige18}). As for the generalization of this estimate for \eqref{eq. TFF}, we expect that if $\kappa-\kappa^{*}>0$ is sufficiently small, then
		\[
		\log T_{\alpha}[\kappa \phi_A] \simeq \kappa^{-\frac{2}{N}}
		\] 
		follows. However, it is hopeless to justify this expectation since the specific value of $\kappa^*$ is unknown, and we cannot observe the behavior of $T_{\alpha}[\kappa \phi_A]$ near $\kappa = \kappa^*$. We do not even know whether $T_{\alpha}[\kappa \phi_A]\to \infty$ as $\kappa \searrow \kappa^*$.
		}
	\end{rem}

	
	\section*{Acknowledgement}
	The authors of this paper would like to thank his supervisor Michiaki Onodera for the helpful advice to the preparation of this paper.
	
	\bibliographystyle{plain}
	\bibliography{ref}

\begin{bibdiv}
\begin{biblist}

\bib{ARCD2004}{article}{
      author={Arrieta, Jose~M.},
      author={Rodriguez-Bernal, Anibal},
      author={Cholewa, Jan~W.},
      author={Dlotko, Tomasz},
       title={Linear parabolic equations in locally uniform spaces},
        date={2004},
        ISSN={0218-2025,1793-6314},
     journal={Math. Models Methods Appl. Sci.},
      volume={14},
      number={2},
       pages={253\ndash 293},
         url={https://doi.org/10.1142/S0218202504003234},
      review={\MR{2040897}},
}

\bib{BarasPierre85}{article}{
      author={Baras, Pierre},
      author={Pierre, Michel},
       title={Crit\`ere d'existence de solutions positives pour des
  \'{e}quations semi-lin\'{e}aires non monotones},
        date={1985},
        ISSN={0294-1449},
     journal={Ann. Inst. H. Poincar\'{e} Anal. Non Lin\'{e}aire},
      volume={2},
      number={3},
       pages={185\ndash 212},
         url={http://www.numdam.org/item?id=AIHPC_1985__2_3_185_0},
      review={\MR{797270}},
}

\bib{BennettSharpley}{book}{
      author={Bennett, Colin},
      author={Sharpley, Robert},
       title={Interpolation of operators},
      series={Pure and Applied Mathematics},
   publisher={Academic Press, Inc., Boston, MA},
        date={1988},
      volume={129},
        ISBN={0-12-088730-4},
      review={\MR{928802}},
}

\bib{BreCaz96}{article}{
      author={Brezis, Ha\"{\i}m},
      author={Cazenave, Thierry},
       title={A nonlinear heat equation with singular initial data},
        date={1996},
        ISSN={0021-7670},
     journal={J. Anal. Math.},
      volume={68},
       pages={277\ndash 304},
         url={https://doi.org/10.1007/BF02790212},
      review={\MR{1403259}},
}

\bib{CarFerNoe24}{article}{
      author={Cort\'{a}zar, Carmen},
      author={Quir\'{o}s, Fernando},
      author={Wolanski, Noem\'{i}},
       title={A semilinear problem associated to the space-time fractional heat
  equation in $\mathbb{R}^n$},
        date={to appear},
     journal={Calc. Var. Partial Differential Equations},
}

\bib{FHIL23}{article}{
      author={Fujishima, Yohei},
      author={Hisa, Kotaro},
      author={Ishige, Kazuhiro},
      author={Laister, Robert},
       title={Solvability of superlinear fractional parabolic equations},
        date={2023},
        ISSN={1424-3199},
     journal={J. Evol. Equ.},
      volume={23},
      number={1},
       pages={38},
        note={Id/No 4},
}

\bib{FujiIoku18}{article}{
      author={Fujishima, Yohei},
      author={Ioku, Norisuke},
       title={Existence and nonexistence of solutions for the heat equation
  with a superlinear source term},
        date={2018},
        ISSN={0021-7824},
     journal={J. Math. Pures Appl. (9)},
      volume={118},
       pages={128\ndash 158},
         url={https://doi.org/10.1016/j.matpur.2018.08.001},
      review={\MR{3852471}},
}

\bib{FujiIshiKawa24}{article}{
      author={Fujishima, Yohei},
      author={Ishige, Kazuhiro},
      author={Kawakami, Tatsuki},
       title={Existence of solutions for a semilinear parabolic system with
  singular initial data},
        date={2024},
      eprint={2407.02847},
         url={https://arxiv.org/abs/2407.02847},
}

\bib{Fujita66}{article}{
      author={Fujita, Hiroshi},
       title={On the blowing up of solutions of the {C}auchy problem for
  {$u_{t}=\Delta u+u^{1+\alpha }$}},
        date={1966},
        ISSN={0368-2269},
     journal={J. Fac. Sci. Univ. Tokyo Sect. I},
      volume={13},
       pages={109\ndash 124 (1966)},
      review={\MR{214914}},
}

\bib{GalWar20}{book}{
      author={Gal, Ciprian~G.},
      author={Warma, Mahamadi},
       title={Fractional-in-time semilinear parabolic equations and
  applications},
      series={Math\'{e}matiques \& Applications (Berlin) [Mathematics \&
  Applications]},
   publisher={Springer, Cham},
        date={[2020] \copyright 2020},
      volume={84},
        ISBN={978-3-030-45043-4; 978-3-030-45042-7},
         url={https://doi.org/10.1007/978-3-030-45043-4},
      review={\MR{4167508}},
}

\bib{GMS22}{article}{
      author={Ghergu, Marius},
      author={Miyamoto, Yasuhito},
      author={Suzuki, Masamitsu},
       title={Solvability for time-fractional semilinear parabolic equations
  with singular initial data},
        date={2022},
        ISSN={0362-546X},
     journal={Math. Meth. Appl. Sci.},
      volume={46},
       pages={6686\ndash 6704},
      review={\MR{3830724}},
}

\bib{Hayakawa73}{article}{
      author={Hayakawa, Kantaro},
       title={On nonexistence of global solutions of some semilinear parabolic
  differential equations},
        date={1973},
        ISSN={0021-4280},
     journal={Proc. Japan Acad.},
      volume={49},
       pages={503\ndash 505},
      review={\MR{338569}},
}

\bib{Hisa21}{incollection}{
      author={Hisa, Kotaro},
       title={Sharp estimate of the life span of solutions to the heat equation
  with a nonlinear boundary condition},
        date={[2021] \copyright 2021},
   booktitle={Geometric properties for parabolic and elliptic {PDE}s},
      series={Springer INdAM Ser.},
      volume={47},
   publisher={Springer, Cham},
       pages={127\ndash 149},
      review={\MR{4279618}},
}

\bib{HisaIshige18}{article}{
      author={Hisa, Kotaro},
      author={Ishige, Kazuhiro},
       title={Existence of solutions for a fractional semilinear parabolic
  equation with singular initial data},
        date={2018},
        ISSN={0362-546X},
     journal={Nonlinear Anal.},
      volume={175},
       pages={108\ndash 132},
         url={https://doi.org/10.1016/j.na.2018.05.011},
      review={\MR{3830724}},
}

\bib{HIT18}{article}{
      author={Hisa, Kotaro},
      author={Ishige, Kazuhiro},
      author={Takahashi, Jin},
       title={Initial traces and solvability for a semilinear heat equation on
  a half space of {$\Bbb R^N$}},
        date={2023},
        ISSN={0002-9947,1088-6850},
     journal={Trans. Amer. Math. Soc.},
      volume={376},
      number={8},
       pages={5731\ndash 5773},
      review={\MR{4630758}},
}

\bib{HisaKojima24}{article}{
      author={Hisa, Kotaro},
      author={Kojima, Mizuki},
       title={On solvability of a time-fractional semilinear heat equation, and
  its quantitative approach to the classical counterpart},
        date={to appear},
     journal={Commun. Pure Appl. Anal.},
}

\bib{IokuIshigeKawakami}{article}{
      author={Ioku, Norisuke},
      author={Ishige, Kazuhiro},
      author={Kawakami, Tatsuki},
       title={Existence of solutions to a fractional semilinear heat equation
  in uniformly local weak zygmund type spaces},
        date={to appear},
     journal={Anal. PDE},
}

\bib{IshiKawaOka2020}{article}{
      author={Ishige, Kazuhiro},
      author={Kawakami, Tatsuki},
      author={Okabe, Shinya},
       title={Existence of solutions for a higher-order semilinear parabolic
  equation with singular initial data},
        date={2020},
        ISSN={0294-1449,1873-1430},
     journal={Ann. Inst. H. Poincar\'{e} C Anal. Non Lin\'{e}aire},
      volume={37},
      number={5},
       pages={1185\ndash 1209},
         url={https://doi.org/10.1016/j.anihpc.2020.04.002},
      review={\MR{4138231}},
}

\bib{IshiKawaTakada24}{article}{
      author={Ishige, Kazuhiro},
      author={Kawakami, Tatsuki},
      author={Takada, Ryo},
       title={Existence of solutions to the fractional semilinear heat equation
  with a singular inhomogeneous term},
        date={2024},
      eprint={2407.17769},
         url={https://arxiv.org/abs/2407.17769},
}

\bib{KoSiTa77}{article}{
      author={Kobayashi, Kusuo},
      author={Sirao, Tunekiti},
      author={Tanaka, Hiroshi},
       title={On the growing up problem for semilinear heat equations},
        date={1977},
        ISSN={0025-5645,1881-1167},
     journal={J. Math. Soc. Japan},
      volume={29},
      number={3},
       pages={407\ndash 424},
         url={https://doi.org/10.2969/jmsj/02930407},
      review={\MR{450783}},
}

\bib{KozoYama94}{article}{
      author={Kozono, Hideo},
      author={Yamazaki, Masao},
       title={Semilinear heat equations and the {Navier}-{Stokes} equation with
  distributions in new function spaces as initial data},
        date={1994},
        ISSN={0360-5302},
     journal={Commun. Partial Differ. Equations},
      volume={19},
      number={5-6},
       pages={959\ndash 1014},
}

\bib{LaiSie21}{article}{
      author={Laister, Robert},
      author={Sier{\.z}{\k{e}}ga, Miko{\l}aj},
       title={A blow-up dichotomy for semilinear fractional heat equations},
        date={2021},
        ISSN={0025-5831},
     journal={Math. Ann.},
      volume={381},
      number={1-2},
       pages={75\ndash 90},
}

\bib{LeeNi92}{article}{
      author={Lee, Tzong-Yow},
      author={Ni, Wei-Ming},
       title={Global existence, large time behavior and life span of solutions
  of a semilinear parabolic {C}auchy problem},
        date={1992},
        ISSN={0002-9947},
     journal={Trans. Amer. Math. Soc.},
      volume={333},
      number={1},
       pages={365\ndash 378},
         url={https://doi.org/10.2307/2154114},
      review={\MR{1057781}},
}

\bib{LuchYama19}{incollection}{
      author={Luchko, Yuri},
      author={Yamamoto, Masahiro},
       title={Maximum principle for the time-fractional {PDE}s},
        date={2019},
   booktitle={Handbook of fractional calculus with applications. {V}ol. 2},
   publisher={De Gruyter, Berlin},
       pages={299\ndash 325},
      review={\MR{3965399}},
}

\bib{MaeTera06}{article}{
      author={Maekawa, Yasunori},
      author={Terasawa, Yutaka},
       title={The {Navier}-{Stokes} equations with initial data in uniformly
  local {{\(L^p\)}} spaces.},
        date={2006},
        ISSN={0893-4983},
     journal={Differ. Integral Equ.},
      volume={19},
      number={4},
       pages={369\ndash 400},
}

\bib{Mainardi1994}{incollection}{
      author={Mainardi, Francesco},
       title={On the initial value problem for the fractional diffusion-wave
  equation},
        date={1994},
   booktitle={Waves and stability in continuous media ({B}ologna, 1993)},
      series={Ser. Adv. Math. Appl. Sci.},
      volume={23},
   publisher={World Sci. Publ., River Edge, NJ},
       pages={246\ndash 251},
      review={\MR{1320083}},
}

\bib{Miya21}{article}{
      author={Miyamoto, Yasuhito},
       title={A doubly critical semilinear heat equation in the {$L^1$} space},
        date={2021},
        ISSN={1424-3199},
     journal={J. Evol. Equ.},
      volume={21},
      number={1},
       pages={151\ndash 166},
         url={https://doi.org/10.1007/s00028-020-00573-2},
      review={\MR{4238200}},
}

\bib{OkaZhan23}{article}{
      author={Oka, Yusuke},
      author={Zhanpeisov, Erbol},
       title={Existence of solutions for time fractional semilinear parabolic
  equations in besov--morrey spaces},
        date={2023},
      eprint={2305.05969},
}

\bib{Oneil63}{article}{
      author={O'Neil, Richard},
       title={Convolution operators and {$L(p,\,q)$} spaces},
        date={1963},
        ISSN={0012-7094,1547-7398},
     journal={Duke Math. J.},
      volume={30},
       pages={129\ndash 142},
         url={http://projecteuclid.org/euclid.dmj/1077374532},
      review={\MR{146673}},
}

\bib{Podl99}{book}{
      author={Podlubny, Igor},
       title={Fractional differential equations},
      series={Mathematics in Science and Engineering},
   publisher={Academic Press, Inc., San Diego, CA},
        date={1999},
      volume={198},
        ISBN={0-12-558840-2},
        note={An introduction to fractional derivatives, fractional
  differential equations, to methods of their solution and some of their
  applications},
      review={\MR{1658022}},
}

\bib{QuitSoup19}{book}{
      author={Quittner, Pavol},
      author={Souplet, Philippe},
       title={Superlinear parabolic problems. {Blow}-up, global existence and
  steady states},
     edition={2nd revised and updated edition},
      series={Birkh{\"a}user Adv. Texts, Basler Lehrb{\"u}ch.},
   publisher={Cham: Birkh{\"a}user},
        date={2019},
        ISBN={978-3-030-18220-5; 978-3-030-18222-9},
}

\bib{RobSier2013}{article}{
      author={Robinson, James~C.},
      author={Sier\.{z}\polhk{e}ga, Miko\l~aj},
       title={Supersolutions for a class of semilinear heat equations},
        date={2013},
        ISSN={1139-1138,1988-2807},
     journal={Rev. Mat. Complut.},
      volume={26},
      number={2},
       pages={341\ndash 360},
         url={https://doi.org/10.1007/s13163-012-0108-9},
      review={\MR{3068603}},
}

\bib{Takahashi16}{incollection}{
      author={Takahashi, Jin},
       title={Solvability of a semilinear parabolic equation with measures as
  initial data},
        date={2016},
   booktitle={Geometric properties for parabolic and elliptic {PDE}'s},
      series={Springer Proc. Math. Stat.},
      volume={176},
   publisher={Springer, [Cham]},
       pages={257\ndash 276},
         url={https://doi.org/10.1007/978-3-319-41538-3_15},
      review={\MR{3571832}},
}

\bib{Wei80}{article}{
      author={Weissler, Fred~B.},
       title={Local existence and nonexistence for semilinear parabolic
  equations in {$L^{p}$}},
        date={1980},
        ISSN={0022-2518},
     journal={Indiana Univ. Math. J.},
      volume={29},
      number={1},
       pages={79\ndash 102},
         url={https://doi.org/10.1512/iumj.1980.29.29007},
      review={\MR{554819}},
}

\bib{ZhanSun15}{article}{
      author={Zhang, Quan-Guo},
      author={Sun, Hong-Rui},
       title={The blow-up and global existence of solutions of {C}auchy
  problems for a time fractional diffusion equation},
        date={2015},
        ISSN={1230-3429},
     journal={Topol. Methods Nonlinear Anal.},
      volume={46},
      number={1},
       pages={69\ndash 92},
         url={https://doi.org/10.12775/TMNA.2015.038},
      review={\MR{3443679}},
}

\bib{ZLS19}{article}{
      author={Zhang, Quanguo},
      author={Li, Yaning},
      author={Su, Menglong},
       title={The local and global existence of solutions for a time fractional
  complex {G}inzburg-{L}andau equation},
        date={2019},
        ISSN={0022-247X},
     journal={J. Math. Anal. Appl.},
      volume={469},
      number={1},
       pages={16\ndash 43},
         url={https://doi.org/10.1016/j.jmaa.2018.08.008},
      review={\MR{3857509}},
}

\end{biblist}
\end{bibdiv}

\end{document}